\documentclass[10pt,a4paper]{amsart}
\usepackage{amscd,amssymb,amsopn,amsmath,amsthm,graphics,amsfonts,enumerate,verbatim,calc}
\usepackage[dvips]{graphicx}\usepackage{booktabs} 

\usepackage{mathpazo}
\usepackage{color}
\usepackage{latexsym}
\usepackage{amsthm,amsfonts,amssymb,mathrsfs}
\usepackage{rotating}
\usepackage[leqno]{amsmath}
\usepackage{xspace}
\usepackage[all]{xy}
\usepackage{longtable}
\headsep=5 mm \numberwithin{equation}{section}
\newtheorem{theorem}{Theorem}[section]
\newtheorem{lemma}[theorem]{Lemma}
\newtheorem{notation}[theorem]{Notation}
\newtheorem{proposition}[theorem]{Proposition}
\newtheorem{corollary}[theorem]{Corollary}
\newtheorem{claim}[theorem]{Claim}
\newtheorem{problem}[theorem]{Problem}
\theoremstyle{definition}
\newtheorem{definition}[theorem]{Definition}
\theoremstyle{remark}
\newtheorem{remark}[theorem]{Remark}
\newtheorem{fact}[theorem]{Fact}
\newtheorem{example}[theorem]{Example}
\newtheorem{observation}[theorem]{Observation}
\newtheorem{discussion}[theorem]{Discussion}

\newtheorem{acknowledgement}{Acknowledgement}

\newcommand{\Ass}{\operatorname{Ass}}
\newcommand{\im}{\operatorname{im}}

\newcommand{\grade}{\operatorname{grade}}

\newcommand{\Spec}{\operatorname{Spec}}

\newcommand{\rad}{\operatorname{rad}}

\newcommand{\Ht}{\operatorname{ht}}
\newcommand{\pd}{\operatorname{p.dim}}\newcommand{\Q}{\mathbb Q}

\newcommand{\Gdim}{\operatorname{Gdim}}

\newcommand{\V}{\operatorname{V}}

\newcommand{\id}{\operatorname{id}}\newcommand{\Ext}{\operatorname{Ext}}

\newcommand{\coker}{\operatorname{coker}}

\newcommand{\tr}{\operatorname{tr}}

\newcommand{\Supp}{\operatorname{Supp}}
\newcommand{\eSupp}{\operatorname{hSupp}}

\newcommand{\Hom}{\operatorname{Hom}}
\newcommand{\Gor}{\operatorname{Gor}}
\newcommand{\Ann}{\operatorname{Ann}}\newcommand{\Att}{\operatorname{Att}}
\newcommand{\Hann}{\operatorname{h-Ann}}

\newcommand{\depth}{\operatorname{depth}}

\newcommand{\hass}{\operatorname{h-Ass}}
\newcommand{\eass}{\operatorname{E-Ass}}
\newcommand{\hid}{\operatorname{Hidd}}

\newcommand{\Min}{\operatorname{Min}}

\newcommand{\lo}{\longrightarrow}
\newcommand{\fm}{\frak{m}}
\newcommand{\fp}{\frak{p}}
\newcommand{\fq}{\frak{q}}
\newcommand{\fa}{\frak{a}}
\newcommand{\fb}{\frak{b}}
\newcommand{\fc}{\frak{c}}

\newcommand{\fn}{\frak{n}}

\newenvironment{PROOF}[2][\proofname.]
{\begin{proof}[#1]\renewcommand{\qedsymbol}{\textsquare$_{\rm #2}$}}
	{\end{proof}}
\begin{document}

\author[]{mohsen asgharzadeh}
\dedicatory {Dedicated to the memory of Wolmer V.\ Vasconcelos}
\address{}
\email{mohsenasgharzadeh@gmail.com}

\title[ ]
{homological subsets of $\Spec$}


\begin{abstract}We recover some data of a module $M$ from  the Ext-family $\{\Ext^i_R(M,R)\}$. In this regard,
we investigate homological subsets of $\Spec(R)$, defined by the help of   Ext-family. For example we present a situation for which $\bigcup_{i=0}^\infty\Ass(\Ext^i_R(M,R))$ is infinite, thereby answering a question of Vasconcelos.
We extend   Grothendieck's calculation of  $\dim(\Ext^g_R(M,R))$ and present some new situations to calculate $\dim(\Ext^\ast_R(M,R))$. Also, we compute support and the set of all
  associated prime ideals of  the Ext-family in a serial of nontrivial cases.
\end{abstract}

\maketitle

\section{Introduction}

Throughout the paper $R$
is a commutative noetherian local ring of dimension $d$,  and $M$ a finitely generated module of grade $g$ and dimension $r$,  otherwise specializes.  This paper deals with the invariants attached to the Ext-modules. Associated prime ideals of
$\Hom_R(-,\sim)$ computed several years ago.  As far as we know, the first computation of $\Ass(\underline{\Ext}^{\textbf{g}_0}_{\ \ \mathcal{O}_X}(\mathcal{F},\mathcal{G}))$  appeared in the $(\textsc{LC})$ by Grothendieck. Here, $(\textsc{LC})$ referred to \textit{local cohomology} \cite{41} (also, see  \cite{sga2}).
Following Vasconcelos, we set $$\eass_R(M):=\bigcup_{i=0}^\infty\Ass(\Ext^i_R(M,R)).$$

\begin{problem}
\begin{enumerate}
	\item[i)] Is $\eass_R(M)$ finite? (Vasconcelos, see \cite{vas1} and \cite{vas}).
	\item[ii)] What is the significance of (the vanishing theorem of) generalized local cohomology? (see \cite[Question 1.1]{moh2}).\end{enumerate}
\end{problem}
In a personal communication,  Vasconcelos kindly informed me that  Problem 1.1(i) were motivated by Huneke's problem on the associated prime ideals of local cohomology modules. In particular, he predicated some negative answer to part (i), and this demonstrates the conjectural connection between  items (i)
and (ii).  We  are going to answer both items, see below.

Note that $\eass_R(M)$ is a subset of another homological set:
We define the \textit{homological support} of a module $(-)$ by $\bigcup_{i}\Supp(\Ext^i_R(-,R))$. 
In \cite{n}  Nunke
showed  a  module over certain Dedekind domains  is zero if and only if its  homological support is empty.
In Section 2  we show  that the homological support have strange properties compared  to the classical support. Despite this, we show
over finitely generated modules,  homological support is the classical support. 
This drops $\pd(-)<\infty$ from an implicit result
of Peskine-Szpiro   \cite[Page 68]{ps} (also see the explicit result of Jans  \cite[Corollary 3.2]{j}), where $\pd_R(-)$ stands for the projective dimension of an $R$-module $(-)$.  This result  has  some  corollaries.
For example, if 
$M$ and $N$ are finitely generated, we show: 
$$
\cup_i \Supp(\Ext^i_R(M,N))=\left(\cup_i \Supp(\Ext^i_R(M,R))\right)\cap\left(\cup_i \Supp(\Ext^i_R(N,R))\right).
$$In particular, $
\cup_i \Supp(\Ext^i_R(M,M))=\cup_i \Supp(\Ext^i_R(M,R))=\Supp(M)$. As an immediate consequence,
$M$ is zero iff $\Ext^i_R(M,R)= 0$ for all $i$ iff $\Ext^i_R(M,M)= 0$ for all $i$.
It may be worth to note that the noncommutative version of the last observation is  interesting:

\textbf{Strong Nakayama conjecture.} Let $A$ be  an Artinian algebra, and $0\neq S$ be  any  simple module. Then
there is an integer $n \geq 0$ such that $\Ext^n_A(S,A)\neq 0$.

For more details, see \cite{ar} by Auslander and Reiten.
Let us  focus on our commutative setting.
As a second application, we extend an implicit result of Grothendieck \cite[6.4.4]{41}  by avoiding
scheme-theory:

\begin{corollary}\label{obdi}
	If $\dim(\Ext^{d-i}_R(M,R))\leq i $ for all $i$ and  $g=d-r$, then $\dim(\Ext^{d-r}_R(M,R))=r.$
\end{corollary}
There are situations for which the assumptions of    Corollary \ref{obdi}    are satisfied, see the discussion after than Corollary \ref{jesse2}. In particular,  under some mild  conditions we have $\dim(\Ext^{d-r}_R(M,R))=r.$
In the light of $(\textsc{LC})$   we observe:

\begin{corollary}(See Proposition \ref{mm})
	The  \textit{dimension formula}  $\dim(\Ext^g_R(M,R))=\dim M$  holds for all $M$ if and only if $R$ is Cohen-Macaulay.
\end{corollary}
If we focus on modules of finite projective dimension over formally equi-dimensional rings, the story will change. In fact,  according to Corollary  \ref{obdi}, we will show that the dimension formula holds for all of such modules. The results in this direction,  started very recently by Beder \cite[Proposition 2.2]{Beder}.
In general, $M$ is not supported in the support of $\Ext^g_R(M,R)$, even over regular rings, see   Example \ref{exa1}.  However, we present situations for which $M$ is supported in  $\Supp(\Ext^g_R(M,R))$, see Corollary \ref{jesse}, \ref{jesse1} and Example \ref{exa}.

Suppose $R$ is a homomorphic image of a Gorenstein ring $S$.
Recall from \cite{vas1}  that the set \textit{homological associated prime ideals} of $M$ is defined by the following set  $\hass_R(M):=\bigcup_i\Ass _R(\Ext^i_S(M,S)).$  This is  a finite set and coincides with the former $\eass_R(M)$ in the Gorenstein case. We  observe that $ \Ass(M)\subseteq \hass(M)$.
In general, $ \Ass(M)\neq\hass(M)$. Despite  this, we present a situation for which $ \Ass(M)= \hass(M)$. Following Vasconcelos, we set  $$M_{(i_1,\ldots,i_p)p}:=\Ext^{i_1}_S(\Ext^{i_2}_S(\ldots(\Ext^{i_p}_S(M,S),\ldots, S),S).$$  This notion introduced by  Vasconcelos  long with his investigation
of the notions of \textit{homological degree} and \textit{well-hidden} associated prime ideals.
We determine $\bigcup_{(i_1,\ldots,i_p)}\Ass(M_{(i_1,\ldots,i_p)p})$ in the diagonal case $\cup_{p=0}^\infty\Ass(M_{(i,\ldots,i)p})=\cup_{p=0}^3\Ass(M_{(i,\ldots,i)p}),$  and we compute it in the Cohen-Macaulay case, see Corollary \ref{refox1} and  \ref{refox2}. 

\S 4 investigates the following different notions of \textit{homological annihilators} and connect them to the classical annihilator:
\begin{enumerate}
	\item[i)] $\beta(M):=\{x\in R:M_x\textit{ is }R_x\textit{-projective}\}$  \ \ \ \ \ \ \ \ \ \ (Auslander-Buchsbaum)
	\item[ii)]  $\alpha(M):=\bigcap_{\Ann(\bigwedge^i(M))\neq  0}\Ann(\bigwedge^i(M)) $\ \ \ \ \ \ \ \ \ \ \ \ \  (Auslander-Buchsbaum)  
	\item[iii)]
	$\gamma(M):=\bigcap _{i>0}\rad(\Ann_R\Ext^i_R(M,R))$\ \ \ \ \ \ \ \ \ \ \ \  (Bridger)	\item[iv)]$\Hann(M):=\prod_{i=0}^d\Ann_R\Ext^i_R(M,R)$\ \ \ \ \ \ \ \ \ \ \  \  (Vasconcelos and others).
\end{enumerate}

\S 5 presents the basic properties of $\eass_R(-)$. It may be $\eass_R(-)=\emptyset$ for a nonzero
module. But, if $(-)$ is finitely generated we show $\eass_R(-)\neq\emptyset$. Against to the classical associated prime ideals,  $\eass_R(-)$ is not
behave well with respect to inclusion (resp. quotient), however we present a situation for which
$\eass(M)\subset\eass(N)\cup\eass(M/N)$ where $N\subset M$. By an explicit computation,
we show $\eass_R(-)\neq\hass_R(-)$ even over maximal Cohen-Macaulay modules. As another corollary of homological support, we show:
\begin{corollary}
	One has $\min(M)\subset\eass(M)$.
\end{corollary}
In particular, similar to the classical associated prime ideals, one has $\Ass(M)\subset\eass(M)$ provided $M$ is equi-dimensional.
According to folklore results, we know $\Ass(-)\subset\Supp(M)$ and their minimal elements are coincide. The corresponding homological version is not true. Despite this and when
$R$ is reduced,  it follows that  $\Ass(R)\cap\eass(M)=\Ass(R)\cap\eSupp(M).$  

\begin{observation}
Concerning Problem 1.1(i), we show:
\begin{enumerate}
	\item[i)] Problem 1.1 reduces to complete case.
	\item[ii)] Problem 1.1 reduces to maximal  Cohen-Macaulay modules, when $R$ is Cohen-Macaulay. \item[iii)] Problem  1.1 is true if $R$ is Gorenstein over the punctured spectrum.
	\item[iv)]Problem 1.1  reduces to cyclic modules when $R$  is normal and Cohen-Macaulay. 
	\item[v)]Problem 1.1 is true over   $3$-dimensional excellent normal local domains.
\end{enumerate}
\end{observation}Also, we present an inductive descent.
In Section 6, we give an explicit commutative   complete Cohen--Macaulay local ring $\mathcal{R}$ of
dimension two and a two-generated Cohen--Macaulay module $M$ for which:
\begin{corollary}
	$\bigcup_{i\geq 0}\Ass_\mathcal{R}\Ext^i_\mathcal{R}(M,\mathcal{R})$ is infinite.
Also,
	\(
	\bigcup_{i\geq0}\Ass_\mathcal{R} H_J^i(M,\mathcal{R})
	\emph { is infinite.}
	\)
\end{corollary}	
 The algebra $\mathcal{R}$ in this construction motivated from the work of Jorgensen--\c{S}ega \cite{seg}, where they proposed for a different aim.
The final section motivated from a result of  Macaulay in 1904.
First, we recall a  more general version of this by  Serre,  see Theorem \ref{6.1} and \ref{6.2}. These results presented in the Ext-from by Griffith and Evans:
Let $R$ be a regular local ring containing a field and $I$ be a height two prime ideal such
that $\Ext^2_R(R/I,R)$ is cyclic. Then $I$ is two generated. We show that:

\begin{corollary}
	Let $(R,\fm)$ be a Cohen-Macaulay local  ring and $I$ be a  Cohen-Macaulay ideal of height two and of finite projective dimension. Then $\mu(\Ext^2_R(R/I,R))=\mu(I)-1.$
\end{corollary} 
In particular, we  slightly extend \cite[Cor II.5.7]{ps} by Peskine-Szpiro.
For more applications of  the Ext-family $\{\Ext^i_R(M,R)\}$ where $M$ is not necessarily finite, see \cite{moh}.
\section{Homological  support }
 By $\pd_R(-)$  (resp. $\id_R(-)$) we mean
projective (resp. injective) dimension. \begin{definition}
i) By homological support of a module $(-)$ we mean
 $\bigcup_{i\geq0}\Supp(\Ext^i_R(-,R))$ and we denote it by $\eSupp(-)$.

ii) 	We say a module  is homologically nonzero if 
 its homological support is not empty. \end{definition}
Homological support may be empty for modules with quite large support:

\begin{example}\label{commentcomplete}
 Let $R$ be a  complete local integral domain of positive dimension. Let $F$  be the fraction field of $R$. It is shown by Auslander \cite[Page 166]{comment} that
$\Ext^i_R(F,R)=0$ for all $i$. So, $\eSupp(F)=\emptyset\neq\Supp(F)=\Spec(R).$
\end{example}

\begin{example}
 The complete-local assumption is important. It may be worth to note that $\Ext^1_\mathbb{Z}(\mathbb{Q},\mathbb{Z})$ is related to the \textit{ad\`{e}le} groups from
number theory. By accepting continuum hypothesis, one has $\Ext^1_\mathbb{Z}(\mathbb{Q},\mathbb{Z})=\mathbb{R}$ as a vector space over $\mathbb{Q}$. So, $\eSupp_\mathbb{Z}(\mathbb{Q})=\Supp_\mathbb{Z}(\mathbb{Q})$.
\end{example}

Maybe homological support is large against to the classical support:

\begin{example}
Adopt the notation of Example 2.1.  It is  shown  in \cite[Page 166]{comment} that
$R\simeq\Ext^1_R(F/R,R)$. So, $$\eSupp(F/R)=\Spec(R)\supsetneqq \Spec(R)\setminus\{(0)\}=\Supp(F/R).$$
\end{example}

\begin{lemma}\label{fg}
Let $L$ and $N$ be finitely generated and nonzero. Then $\Ext^i_R(L,N)\neq 0$   for some $i\leq \dim N$.
\end{lemma}

\begin{proof}  For each ideal $I$ recall that $\Ht_N(I)$ were defined by $\inf\{\dim(N_{\fp}):\fp\in\Supp(N)\cap\V(I)\}$, where $\V(I)$ is the set of all prime ideals containing $I$. Since $R$ is local,  we have $L\otimes N\neq  0$. Consequently, $\Supp(L)\cap\Supp(N)\neq\emptyset$, i.e., $\Supp(N)\cap\V(\Ann L)\neq\emptyset$. Deduce from this
that $\Ht_N(\Ann L)<\infty$.
It sufficient to recall that  $\inf\{i:\Ext^i_R(L,N)\neq 0\}=\grade(\Ann L,N)\leq\Ht_N(\Ann L)<\infty .$ 
\end{proof}

\begin{corollary}\label{eqq}
Keep the above notation in mind. Then
$\Supp(L\otimes N)\subset\bigcup_{i=0}^{\dim N} \Supp(\Ext^i_R(L,N))$. 
\end{corollary}

\begin{proof}
Let $\fp\in\Supp(L\otimes N)$. Then $L_{\fp}$ and $N_{\fp}$ are nonzero. In view of Lemma \ref{fg}, there is an $i\leq \dim N_{\fp}\leq \dim N$ such that
$\Ext^i_{R_{\fp}}(L_{\fp},N_{\fp})\neq 0$.  Note that $\Ext^i_R(L,N)_{\fp}\simeq \Ext^i_{R_{\fp}}(L_{\fp},N_{\fp})$, because of the finiteness of $L$ and $N$. So, $\fp\in \bigcup_{i=0}^{\dim N}\Supp(\Ext^i_R(L,N))$.
\end{proof}

\begin{proposition} \label{f} Let $M$ and $N$ be  finitely generated and nonzero. Then
$$\Supp(M\otimes N)=\bigcup_{i=0}^{\infty} \Supp(\Ext^i_R(M,N))=\bigcup_{i=0}^{\dim N} \Supp(\Ext^i_R(M,N)).$$In particular, $\eSupp(M)=\Supp(M).$
\end{proposition}

\begin{proof}
 We bring the following  trivial facts  \begin{enumerate}
\item[1)] $\Supp(M\otimes N)=\Supp(M)\cap \Supp(N)$,
\item[2)]  $\Supp(\Ext^i_R(M,N))\subset\Supp(M)\cap \Supp(N)$, and
\item[3)]
$\cup_{i=0}^{\dim N} \Supp(\Ext^i_R(M,N))\subset\cup_i \Supp(\Ext^i_R(M,N))$.
\end{enumerate}
We look at
$$
\begin{array}{ll}
\Supp(M\otimes N)&\stackrel{1}=\Supp(M)\cap \Supp(N)\stackrel{2}\supseteq\bigcup_i \Supp(\Ext^i_R(M,N))\\&\stackrel{3}\supseteq\bigcup_{i=0}^{\dim N} \Supp(\Ext^i_R(M,N))
\stackrel{\ref{eqq}}\supseteq\Supp(M\otimes N).
\end{array}$$
The proof is now complete.
\end{proof}

 $M$ is called \textit{quasi-perfect} if $\inf\{i:\Ext^i_R(M,R)\neq 0\}=\sup\{i:\Ext^i_R(M,R)\neq 0\}.$

\begin{corollary}\label{jesse}
Let $M$ be quasi-perfect of grade $g$. Then $\Supp(M)=\Supp(\Ext^{g}_R(M,R))$.
In particular, $\dim(M)=\dim(\Ext^{g}_R(M,R))$.
\end{corollary}

\begin{proof}
	This is immediate from Proposition
	\ref{f}.
\end{proof}

\begin{corollary}\label{jesse}
	Let $M$ and $N$ be finitely generated. Then $$
	\begin{array}{ll}
	\bigcup_i \Supp(\Ext^i_R(M\otimes_RN,R))&=\bigcup_i \Supp(\Ext^i_R(M,N))\\&=\left(\bigcup_i \Supp(\Ext^i_R(M,R))\right)\cap\left(\bigcup_i \Supp(\Ext^i_R(N,R))\right).
	\end{array}$$In particular, $
	\cup_i \Supp(\Ext^i_R(M,M))=\cup_i \Supp(\Ext^i_R(M,R))$.
\end{corollary}

\begin{proof}
Since $M$ and $N$ are finitely generated, 	we have $\Supp(M\otimes N)\stackrel{(+)}=\Supp(M)\cap \Supp(N)$.
We apply this along with Proposition
\ref{f}, and deduce
$$
\begin{array}{ll}
\left(\bigcup_i \Supp(\Ext^i_R(M,R))\right)\cap\left(\bigcup_i \Supp(\Ext^i_R(N,R))\right)&\stackrel{\ref{f}}=\Supp(M)\cap \Supp(N)\\&\stackrel{(+)}=\Supp(M\otimes N)\\&\stackrel{\ref{f}}=\bigcup_{i}\Supp(\Ext^i_R(M,N)).
\end{array}$$
Another use of Proposition
\ref{f}, completes the proof of first claim.
To see the particular case we put $N:=M$ and apply the first part.
\end{proof}

\begin{corollary}Let $M$ be finitely  generated. The following assertions are equivalent:
	\begin{enumerate}
		\item[1)]  $M$ is zero,
\item[2)] 
$\Ext^i_R(M,R)= 0$ for all $i$,
\item[3)]  $\Ext^i_R(M,R)= 0$ for all $i\leq \dim R$,
\item[4)] $\Ext^i_R(M,M)= 0$ for all $i$, 
	\item[5)]  $\Ext^i_R(M,M)= 0$ for all $i\leq \dim R$.
\end{enumerate}
\end{corollary}

\begin{remark}We say a module  $F$ is homologically  free if 
	$\Ext^i_R(F,R)=\Ext^i_R(F,F)=0$ for all $i>0$. By this terminology the following are valid.
\begin{enumerate}
\item[1)]  	Let $R$ be a 1-dimensional comolete local integral domain.
	Then the fraction field of $R$  is homologically  free. 
	
\item[2)] 
	In item 1), the complete
	assumption is important. 
	
\item[3)]  A conjecture of Auslander and Reiten predicted that 
	finitely generated  homologically  free modules over a local integral domain
	is free.\end{enumerate}
\end{remark}

\begin{corollary}\label{jesse1}
Let $(R,\fm)$ be a $d$-dimensional Cohen-Macaulay local ring, Gorenstein on its punctured spectrum and $M$ be  1-dimensional.  Then $\Supp(\Ext^{d-1}_R(M,R))=\Supp(M)$.  In  particular, $\dim(\Ext^{d-1}_R(M,R))=1$.
\end{corollary}

\begin{proof}
We have $\id(R_{\fp})=\Ht(\fp)$ for all non-maximal prime ideal $\fp$. This yields  that $\Supp(\Ext^d_R(M,R))\subset \{\fm\}$. We are going to apply the grade conjecture. This is the place that we use the
Cohen-Macaulay assumption, because  the conjecture verified over such a ring without any stress on the finiteness of projective dimension, see \cite[I. Lemma 4.8]{ps}. Thus,  $\grade(M)=\dim R-\dim M=d-1,$ e.g. $\Ext^{d-1}_R(M,R)\neq0$. In particular, $\fm$ belongs to its support. In view of Proposition  \ref{f}, we have $$
\begin{array}{ll}
\Supp(M)&=\Supp(\Ext^{d-1}_R(M,R))\cup\Supp(\Ext^{d}_R(M,R))\\&=\Supp(\Ext^{d-1}_R(M,R))\cup\{\fm\}=\Supp(\Ext^{d-1}_R(M,R)).\\
\end{array}$$
This is what we want to prove.
\end{proof}

\begin{corollary}\label{jesse2}
Let $(R,\fm)$ be a $d$-dimensional local ring and $M$ be $r$-dimensional. If $\dim(\Ext^{d-i}_R(M,R))\leq i $ for all $i$ and that $g=d-r$,  then $\dim(\Ext^{g}_R(M,R))=r.$
\end{corollary}

 \begin{proof}
  Set $\fb_i:=\Ann(\Ext^{i}_R(M,R))$.
 In view of Proposition  \ref{f} we see  $
\Supp(M)=\V(\fb_{g})\cup\cdots\cup\V(\fb_{d}).$  Now, we look at the    strict and maximal  chain $\fp_0\subsetneqq \cdots\subsetneqq\fp_r=\fm $ of prime ideals in the $\Supp(M)$.
We claim that $\fp_0\in\V(\fb_{g})$. Indeed, suppose on the way of  contradiction that  $\fp_0\notin\V(\fb_{g})$. Hence $\fp_0\in\V(\fb_{\ell})$ for some $\ell>g$, i.e., $\fp_0\supset \fb_{\ell}$. Clearly, $\fp_j\in\V(\fb_{\ell})$ for all $j$. By definition of $\dim$, we have $\dim (R/\fb_{\ell})\geq r=d-g$. Due to our assumption, $\V(\fb_{\ell})$ is of dimension at most $d-\ell$. Combining these, we observe $d-g>d-\ell\geq\dim (R/\fb_{\ell})\geq d-g,$ which is a contradiction.
Thus, $\fp_0\in\V(\fb_{g})$. Therefore,
$r=\dim( M)\geq\dim(\Ext^{g}_R(M,R))\geq r.$ So,  $\dim(\Ext^{g}_R(M,R))=r$.
\end{proof}

\begin{remark}
Adopt the previous notion and assumptions.	Its  first condition holds when $\pd(M)<\infty$ or $R$ is Gorenstein. The second condition holds either $R$ is Cohen-Macaulay or $\pd(M)<\infty$ and $R$ is formally equidimensional.
\end{remark}
 The following shows that $\Supp(\Ext^{g}_R(M,R))\neq\Supp(M)$.
\begin{example}\label{exa1}
	 Let $R:=\mathbb{Q}[[x,y,z]]$  and let $I:=(xy,xz)$.  Then
		$\Supp(R/I)\neq\Supp(\Ext^{g}_R(R/I,R))$.
\end{example}

\begin{proof}
	 Set $M:=R/I$. We look at the minimal free resolution of $M$:$$\begin{CD}
	0 @>>> R @>
	\begin{array}{ccc}
	[-z\quad
	y ]\\
	\end{array}>>  R^2 @>[
	xy   \quad
	xz ]^t>>  R.
	\\
	\end{CD}
	$$
	Apply $\Hom_R(-,R)$ to it, we get to the following complex $$\begin{CD}
	R @>[
	xy   \quad
	xz ]^t
	>>  R^2 @>
	[-z\quad y ]>>  R @>>> 0.
	\\
	\end{CD}
	$$
	Note that $\Ass(\Hom_R(M,R))=\Supp(M)\cap \Ass(R)=\emptyset$. Let us compute the $\Ext^1_R(M,R)$.
	As $y,z$ is a regular sequence in $R$, the \textit{Koszul} complex on $\{y,z\}$ presents the free resolution of  $(y,z)$. Thus, the kernel of  $(-z,y):R^2\stackrel{}\lo R$ is generated
	by  $(y,z)$ and so $\Ext^1_R(M,R)\simeq (y,z)R/(xy,xz)R.$
	This yields that $\Ann(\Ext^1_R(M,R))=(x)$.
	Therefore, $g:=\grade(M)=1$  and  $(y,z)\in \Supp(M)\setminus \V(x)$. We deduce from  this that $\Supp(\Ext^g_R(M,R))=\V(x)\subsetneqq \Supp(M).$
\end{proof}

\begin{example}\label{exa}Let $R$ be Gorenstein. The following holds:
\begin{enumerate}
	\item[i)]  $L:=\Ext^{i}_R(\Ext^{i}_R(M,R),R)$ for each $i$. Then $\Supp(\Ext^{i}_R(L,R))=\Supp(L)$.

	\item[ii)] Let $R$ be Gorenstein  and $M$ a Cohen-Macaulay module of grade $g$. Then $\Supp(\Ext^{g}_R(M,R))=\Supp(M)$.\end{enumerate}
\end{example}

\begin{proof}
i) This follows by \cite[7.60]{brid}, where Bridger has shown the following  result: $$L=\Ext^{i}_R(\Ext^{i}_R(M,R),R)\simeq\Ext^{i}_R(\Ext^{i}_R(\Ext^{i}_R(\Ext^{i}_R(M,R),R),R),R)=\Ext^{i}_R(\Ext^{i}_R(L,R),R).$$

ii) This follows from the Ext-duality \cite[3.3.10]{BH}:  $M\simeq \Ext^{g}_R(\Ext^{g}_R(M,R),R).$
\end{proof}

\begin{discussion}\label{dim1}(Grothendieck \cite[6.4.4)]{41})
Let $(R,\fm)$ be a $d$-dimensional Cohen-Macaulay local ring with a canonical module and $M$ a finitely generated module. Then $\dim(\Ext^{d-\dim M}_R(M,\omega_R))=\dim M$.
\end{discussion}

\begin{proposition}\label{mm}
Let $(R,\fm)$ be  a local ring and $M$ a finitely generated module of grade $g$. Then $\dim(\Ext^g(M,R))=\dim M$  for all $M$ if and only if $R$ is Cohen-Macaulay.
\end{proposition}

\begin{proof}
Suppose first that  $R$  is not Cohen-Macaulay. We show the dimension formula does not hold   for the residue field. To this end, we just note that $g:=\grade(R/\fm)=\depth(R)$. So  $\dim(\Ext^{g}_R(R/\fm,R))=0\lneqq d-g.$

Suppose now that  $R$  is  Cohen-Macaulay. In order to prove the desired formula without loss of generality, we assume that $R$ is complete. Then $R$ has a canonical module \cite[Theorem 3.3.6]{BH}. By Discussion \ref{dim1}, $\dim(\Ext^{d-\dim M}_R(M,\omega_R))=\dim M$. Revisiting \cite[Lemma I.4.8]{ps}, one has $g=d-\dim M$.
Let $\underline{x}:=x_1,\ldots,x_g$ be a maximal $R$-sequence in $\Ann(M)$. Put $\overline{R}:=R/\underline{x}R$ and note that $M$ can view as an $\overline{R}$-module. As $\omega_R$ is maximal Cohen-Macaulay,  $\underline{x}$ is a $\omega_R$-sequence in $\Ann(M)$. Due to the Rees lemma \cite[Page 140]{mat}, there is the isomorphism $\Ext^{g}_R(M,\omega_R)\simeq \Ext^{0}_{\overline{R}}(M,\omega_R/\underline{x}\omega_R)\simeq \Ext^{0}_{\overline{R}}(M,\omega_{\overline{R}}),$  where the last isomorphism deduces from \cite[Theorem 3.3.5]{BH}, e.g., $\omega_{\overline{R}}\simeq \omega_R/\underline{x}\omega_R$. In a similar vein
there is an isomorphism $\Ext^{g}_R(M,R)\simeq  \Ext^{0}_{\overline{R}}(M,\overline{R}).$

Claim:  One has  $\dim(\Hom_{\overline{R}} (M,\overline{R}))=\dim(\Hom _{\overline{R}} (M,\omega_{\overline{R}}))$. To this end we show they have the same associated prime ideals. As $\overline{R}$ and $\omega_{\overline{R}}$
are Cohen-Macaulay over $\overline{R}$, their associated prime ideals are the minimal primes of their support. On the other hand $\overline{R}$ and $\omega_{\overline{R}}$
have same  support, because $(\omega_{\overline{R}})_{\fp}=\omega_{\overline{R}_{\fp}}$, see \cite[Theorem 3.3.5]{BH}.
Thus, $\Ass_{\overline{R}}(\overline{R})=\Ass_{\overline{R}}(\omega_{\overline{R}})$. In view of \cite[Sublemma 3.2]{41}, $$\Ass_{\overline{R}}(\Hom_{\overline{R}}(M,\overline{R}))=\Supp_{\overline{R}}(M)\cap\Ass_{\overline{R}}(\overline{R})=\Supp_{\overline{R}}(M)\cap\Ass_{\overline{R}}(\omega_{\overline{R}})=
\Ass_{\overline{R}}(\Hom_{\overline{R}}(M,\omega_{\overline{R}})). $$

Combining these we get$$
\begin{array}{ll}
\dim_R (M)&= \dim(\Ext^{g}_R(M,\omega_R)) =\dim(\Hom _{\overline{R}} (M,\omega_{\overline{R}}))\\&=\dim(\Hom_{\overline{R}} (M,\overline{R})) =\dim(\Ext^g_R(M,R)),
\end{array}$$as claimed.
\end{proof}

We say a module $F$ is homologically zero if 
$\Ext^i_R(F,R)=0$ for all $i\geq 0$. 
We say a module is homologically  free if 
$\Ext^i_R(F,R)=\Ext^i_R(F,F)=0$ for all $i>0$.
\section{homological associated primes}

We start by a useful example and a motivational discussion:
\begin{example}\label{exaa}
Let $R:=\mathbb{Q}[[x,y,z]]$  and let $M:=R/(xy,xz)$. Then
$ \cup\Min(\Ext^i_R(M,R))=\Ass_R(M)=\{(x),(y,z)\}.$ 
\end{example}

\begin{proof}
In view of  Example \ref{exa}, we have:\begin{equation*}
\Ext^i_R(M,R)\simeq \left\{
\begin{array}{rl}
\frac{(-y,z)}{(xy,xz)} & \  \   \   \   \   \ \  \   \   \   \   \ \text{if } i=1\\
 \frac{R}{(y,z)}&\  \   \   \   \   \ \  \   \   \   \   \ \text{if } i=2\\
0 & \  \   \   \   \   \ \  \   \   \   \   \ otherwise
\end{array} \right.
\end{equation*}

Thus, $\cup\Min(\Ext^i_R(M,R))=\Ass(M)$.
\end{proof}

\begin{discussion}\label{dim}(Grothendieck)
Let $(R,\fm)$ be a Cohen-Macaulay local ring of dimension $d$ with a canonical module. Let $M$ be a finitely generated module of dimension $n$ and depth $t$. The following  holds.
\begin{enumerate}
\item[i)] $\Ext^i_R(M,\omega_R)=0$ for $i\notin [d-n,d-t]$. Also, $\Ext^i_R(M,\omega_R)\neq0$ for $i=d-n$ and $i=d-t$.
\item[ii)] $\dim(\Ext^i_R(M,\omega_R))\leq d-i$.
\end{enumerate}
\end{discussion}

\begin{proposition}
Let $M$ and $N$ be finitely generated modules such that either $\pd(M)<\infty$ or $\id(N)<\infty$ over any commutative noetherian ring. Let $i\leq d$. Then $\dim(\Ext^{i}_R(M,N))\leq d-i$.
\end{proposition}

\begin{proof}
 Without loss of generality we assume that $d:=\dim R$ is finite.
Suppose $\fp$ is of coheight $>d-i$. Suppose first that $\id(N)<\infty$. Due to a  theorem of Bass \cite[Theorem 3.1.17]{BH},  we observe that
 $$\id(N_{\fp})\leq\depth(R_{\fp})\leq\dim R_{\fp}=\Ht(\fp)\leq d-\dim R/\fp< d-(d-i)=i.$$ Thus $\Ext^{i}_{R_{\fp}}(M_{\fp},N_{\fp})=0$.  Suppose now that $\pd(M)<\infty$. By Auslander-Buchsbaum formula and in a similar way as above, $\pd(M_{\fp})<i$. Consequently, $\Ext^{i}_{R_{\fp}}(M_{\fp},N_{\fp})=0$. We show in each cases that
$\Supp(\Ext^{i}_R(M,N))$ is of dimension less or equal than $d-i$. This proves the desired fact.\end{proof}

\begin{example} The first item says that finitely generated assumption is really needed. The second item says that  finiteness of homological dimension is important:
\begin{enumerate}
	\item[i)] In view of  Example 2.2, $\Ext^1_\mathbb{Z}(\mathbb{Q},\mathbb{Z})$ is a nonzero rational vector space. So, $\dim(\Ext^1_\mathbb{Z}(\mathbb{Q},\mathbb{Z}))=1>0=d-1$.
\item[ii)] The finiteness of homological dimensions is important. Look at $R:=\mathbb{Q}[[X,Y]]/(X^2)$. It is easy to see $ \Ext^i_R(R/xR,R/xR)\simeq R/xR$  for all $i$. So, $\dim (\Ext^1_R(R/xR,R/xR))=1>0=d-1$.\end{enumerate}
\end{example}

In the above example $\Ext^i_R(R/(x),R)=0$ for all $i>0$. One may search the validity of  $\dim(\Ext^{i}_R(-,R))\leq d-i$. In general, this is not the case as the next example says.

\begin{example}
 Let $R:=\mathbb{Q}[X,Y,Z]/(X^2,XY,XZ)$.
 There is a finitely generated $R$-module $M$ such that $\dim(\Ext^{2}_R(M,R))>d-i=\dim R-2$.
 
\end{example}
 
\begin{proof}  This is a 2-dimensional ring and $\min(R)=\{(x)\}$.  Due to the formula $\rad(y,z)=(x,y,z)$, one can show that $\{y,z\}$ is a system of parameters. Since $\{y,z\}$ is not a regular sequence, $R$ is not Cohen-Macaulay. Set $M:=R/xR$. Note that neither $\pd(M)<\infty$ nor $\id(R)<\infty$. We are going to show $\dim (\Ext^2_R(M,R))=2>0=d-2$.
 We restate the 3 relations $x^2=xy=xz=0$ of $\{x,y,z\}$ by
 \begin{equation*}
\left(
\begin{array}{cccccc}
x & 0  & 0 \\
0 & x  & 0  \\
0 & 0  & x
\end{array} \right)\left(
\begin{array}{ccc}
x  \\
y  \\
z
\end{array} \right)=\left(
\begin{array}{ccc}
0  \\
0  \\
0
\end{array} \right).
\end{equation*}
  We restate the   Koszul relations of $\{x,y,z\}$ by  \begin{equation*}
\left(
\begin{array}{cccccc}
-y & -z & 0 \\
x  & 0  & -z\\
0 & x  & y
\end{array} \right)\left(
\begin{array}{ccc}
x  \\
y  \\
z
\end{array} \right)=\left(
\begin{array}{ccc}
0  \\
0  \\
0
\end{array} \right).
\end{equation*}
 Then we put all of these relations of $\{x,y,z\}$ to the following package \begin{equation*}
\textsc{A}:= \left(
\begin{array}{cccccc}
x & 0  & 0 & -y & -z & 0\\
0 & x  & 0 & x  & 0  & -z\\
0 & 0  & x &  0 & x  & y
\end{array} \right).
\end{equation*} Look at the free resolution of $M$:$$\begin{CD}
\cdots @>>> R^6 @>
A
>>  R^3 @>[x \quad
y  \quad
z  ]>>  R@>x>> R @>>> M@>>> 0.@.
\\
\end{CD}
$$Delete $M$ from the right and apply $\Hom_R(-,R)$ we get to the following complex$$\begin{CD}
0 @>>> R @>
x
>>  R @>[
x\quad
y \quad
z
]^t>>  R^3@>\textsc{A}^t>> R^6 .@.
\\
\end{CD}
$$ Let $\{a,b,c\}$ be such that $[a,b,c]\textsc{A}=0$. It is solution of the following system of  six equations \begin{equation*}
\left\{
\begin{array}{rl}
ax=bx=cx=0 & \\
-ay+bx=0 & \\
-az+cx=0 & \\
-bz+cy=0
\end{array} \right.
\end{equation*}It is easy to see that $\{(x,0,0),(0,x,0),(0,0,x),(0,y,z)\}$ are the solutions. Hence $$\Ext^2_R(M,R)=\frac{\ker( A^t)}{\im([x\quad y\quad z]^t)}\supseteq\frac{\langle(x,0,0),(0,x,0),(0,0,x),(0,y,z)\rangle}{(x,y,z)R}.$$
Thus,  $$(x)\subseteqq\Ann(\Ext^2_R(R/(x),R))\subseteq\Ann(\frac{\langle(x,0,0),(0,x,0),(0,0,x),(0,y,z)\rangle}{(x,y,z)R})=(x).$$We observe that $\dim(\Ext^2_R(M,R))=\dim R/(x)=2.$
\end{proof}

\begin{proposition}\label{coh}
Let $(R,\fm)$ be a local ring which is a homomorphic image of a Gorenstein ring
$(S,\fn)$  of dimension $d$. Let $M$ be a finitely generated $R$-module of dimension $n$. Then $\Ass_R(M)\subseteq\bigcup_{i=d-n}^n\min_R(\Ext^i_S(M,S)).$
The equality holds if $M$ is Cohen-Macaulay.
\end{proposition}

\begin{proof}We apply an idea  taking from Grothendieck \cite[Proposition 6.6]{41}.
By $(-)^v$ we mean the Matlis dual functor. We view $M$ as an $S$-module via the map $S\to R$. In the light of the  local duality theorem, we have $H^i_{\fm}(M)^v\simeq\Ext^{d-i}_S(M,S)$. Also, $H^i_{\fm}(M)\simeq H^i_{\fn}(M)$ as an $R$-module. Since $H^i_{\fm}(M)^{vv}\simeq H^i_{\fm}(M)$ we observe that $$\fa_i(M):=\Ann_R (H^i_{\fm}(M))=\Ann_R(H^i_{\fn}(M))=\Ann_R (\Ext^{d-i}_S(M,S)).$$ In view of  Discussion \ref{dim}(ii), $\dim\Ext^{d-i}_S(M,S)\leq i$. Let $\fp\in \Ass_R(M)$ of dimension $i$
and let $\fq$ be its image in $S$. We know $\dim(S_{\fq})=d-i$. By Discussion \ref{dim}(i) $\Ext^{d-i}_{S_{\fq}}(M_{\fq},S_{\fq})\neq 0$ provided $\depth(M_{\fq})=0$. We denote this property by $(\dagger)$.   Then
$$
\begin{array}{ll}
\fp\in \Ass(M)&\Longleftrightarrow \depth(M_{\fp})=0\\&\Longleftrightarrow \depth(M_{\fq})=0\\&\stackrel{(\dagger)}\Longrightarrow \Ext^{d-i}_{S_{\fq}}(M_{\fq},S_{\fq})\neq 0\\&\Longleftrightarrow \Ext^{d-i}_{S}(M,S)_{\fp}\neq 0.
\end{array}$$ 
As minimal elements of support are the associated
primes, we get the first claim by recalling again that $\dim(\Ext^{d-i}_S(M,S))\leq i.$
Now, suppose that $M$ is Cohen-Macaulay. In view of Discussion \ref{dim}(i),
the implication $(\dagger)$ holds in both directions. So, $\Ass_R(M)=\bigcup_{i=d-n}^n\min_R(\Ext^i_S(M,S)).$
\end{proof}

\begin{corollary}\label{att1} Let $M$ be finitely generated. Then
$ \Ass(M)\subseteq \hass(M)$.
\end{corollary}
\begin{enumerate}
	\item[$\bullet$] First proof:  Clear from  Proposition \ref{coh}.
	\item[$\bullet$] Second proof: Here we use well-known results from the theory
	of attach prime ideals and leave the routine details to the reader. To this end, recall that $$\Ass(M)\subset\bigcup_i\Att(H^i_{\fm}(M))=\bigcup_i\Att(\Ext^{d-i}_R(M,R)^v)=
	\bigcup_i\Ass(\Ext^{d-i}_R(M,R))=\hass(M).$$
\end{enumerate}

Recall from \cite{vas} the \textit{hidden associated prime ideals} of $M$ is $\hid_R(M):=\hass_R(M)\setminus \Ass_R(M)$.
Now, we recover  a result of Foxby \cite[Proposition 3.4b]{ff} via three different arguments:

\begin{corollary}\label{refox}
Let $R$ be a Gorenstein local ring and $M$ a Cohen-Macaulay module. Then $\hid(M)=\emptyset$.
\end{corollary}

\begin{enumerate}
	\item[$\bullet$] First proof:  As $M$ is Cohen-Macaulay, $\Ext^i_R(M,R)$ is either zero or Cohen-Macaulay.
	This may be respell by  $\min (\Ext^i_R(M,R))=\Ass_R (\Ext^i_R(M,R))$.
	Proposition \ref{coh} yields the claim.
	\item[$\bullet$] Second proof: Apply Corollary \ref{jesse2}.
	\item[$\bullet$] Third proof: Use the concept of attach prime ideals (see the second proof of Corollary \ref{att1}). We  leave the routine details to the reader.
\end{enumerate}

The Cohen-Macaulay assumption is important:

\begin{example} Let $M$ be nonfree and torsion-free over a local domain. Assume in addition that one of the following conditions hold:
\begin{enumerate}
	\item[i)] $R$ is 2-dimnsional and regular.
	 \item[ii)]$R$ is 3-dimnsional and regular and $M$ is reflexive.
	 \item[iii)] $M$ is locally free over the punctured spectrum.
 \end{enumerate}
Then
$\hid(M)=\{\fm\}$.
\end{example}

\begin{proof}We only  present the proof of i).
It follows that $\pd(M)=1$ and that $M$ is locally free on the punctured spectrum.
From this, $\Ext^{>1}_R(M,R)=0$ and that 	$\Ext^{1}_R(M,R)\neq0$ is of finite
length. Over any  2-dimensional regular local ring, any reflexive module is free.
From this,
$\Hom_R(M,R)$ is free. In sum,
$\cup_i\Ass_R (\Ext^i_R(M,R))=\{0\}\cup\{\fm\}$ and $\Ass(M)=\{0\}$.
Therefore,	$\hid(M)=\{\fm\}$.
\end{proof}

\begin{notation}Recall  that  $M_{(i_1,\ldots,i_p)p}:=\Ext^{i_1}_S(\Ext^{i_2}_S(\ldots(\Ext^{i_p}_S(M,S),\ldots, S),S).$\end{notation}
Here, we drop the later $p$ from $(i_1,\ldots,i_p)p$ if there is no danger of confusion.

\begin{corollary}\label{refox1}
Let $(R,\fm)$ be a Gorenstein  and $M$ be Cohen-Macaulay. Then  $\cup_{p}\cup_{(i_1,\ldots,i_p)}\Ass(M_{(i_1,\ldots,i_p)})=\Ass(M).$ 
\end{corollary}

\begin{proof}
Let $d:=\dim R$ and $e:=\dim M$.	Since $M\cong \Ext^{d-e}_R(\Ext^{d-e}_R(M,R),R)$, one has $\Ass(M)\subseteq\cup_{p}\cup_{(i_1,\ldots,i_p)}\Ass(M_{(i_1,\ldots,i_p)}).$ The reverse inclusion follows from Corollary \ref{refox}.
\end{proof}	
\begin{example}\label{exaaa}We revisit Example \ref{exa}.
Let $S:=\mathbb{Q}[[x,y,z]]$  and let $M:=S/(xy,xz)$. Let $p\geq1$. The following holds:\begin{equation*}
M_{(i_1,\ldots,i_{p-1},2)}:=\left\{
\begin{array}{rl}
S/(y,z) & \text{if } i_1=\ldots=i_{p-1}=2\\
0 & \text{otherwise }
\end{array} \right.
\end{equation*}
\end{example}

\begin{proof}
In view of  Example \ref{exa},
$\Ext^2_S(M,S)\simeq
 \frac{S}{(y,z)}$. Since $y,z$ is  a regular sequence on $S$, from its Koszul complex we obtain that $\Ext^i_S(\Ext^2_S(M,S),S)\simeq
 \frac{S}{(y,z)}\delta_{2,i}$, where $\delta$ is the Kronecker delta. From this we get the claim.
\end{proof}

\begin{corollary}\label{refox2}
Let $R$ be a Gorenstein local ring, $M$ a finitely generated module and $i$ be any integer. Then  $\bigcup_{p=0}^{\infty}\Ass(M_{(i,\ldots,i)p})=\bigcup_{p=0}^3\Ass(M_{(i,\ldots,i)p}).$ 
\end{corollary}

\begin{proof}
	Let $p>3$. We claim that $M_{(i,i,i)}\simeq M_{(i,\ldots,i)p}$ if $p$ is odd and $M_{(i,i)}\simeq M_{(i,\ldots,i)p}$ if $p$ is even. We proceed by induction on $i$. Let $p=4$.
It is enough to  revisit  Bridger's amusing formula: $M_{(i,i)}\simeq M_{(i,i,i,i)}$. Let $p=5$. 
We apply $\Ext^i_R(-,R)$ to $M_{(i,i)}\simeq M_{(i,i,i,i)}$ to see
 $M_{(i,i,i)}\simeq M_{(i,i,i,i,i)}$.
Now suppose, inductively, that $p\geq 6$ and assume the claim for all integer  smaller than $p$. 
First, suppose $p$ is odd. By inductive hypothesis,
$M_{(i,i)}\simeq M_{(i,\ldots,i)p-1}$. We apply $\Ext^i_R(-,R)$ to it, to deduce that $M_{(i,i,i)}\simeq M_{(i,\ldots,i)p}$. Finally, suppose  $p$ is even. By induction,
$M_{(i,i,i)}\simeq M_{(i,\ldots,i)p-1}$. We apply $\Ext^i_R(-,R)$ to it, to deduce that $M_{(i,i,i,i)}\simeq M_{(i,\ldots,i)p}$. We combine this along with Bridger's amusing formula
to deduce that $M_{(i,i)}\simeq M_{(i,i,i,i)}\simeq M_{(i,\ldots,i)p}.$   The claim is now proved. In view
of this, we observe that $\bigcup_{p=0}^3\Ass(M_{(i,\ldots,i)p})=\bigcup_{p=0}^{\infty}\Ass(M_{(i,\ldots,i)p}).$
\end{proof}

\section{homological annihilators}
Our aim in this section is to find relation between different notion of homological annihilators and connect them to the classical annihilator. Our list of homological annihilators is as follows:
\begin{enumerate}
	\item[i)] $\beta(M):=\{x\in R:M_x\textit{ is }R_x\textit{-projective}\}$, 
	\item[ii)]  $\alpha(M):=\bigcap_{\Ann(\bigwedge^i(M))\neq  0}\Ann(\bigwedge^i(M)) $,
	\item[iii)]
	$\gamma(M):=\bigcap _{i>0}\rad(\Ann_R\Ext^i_R(M,R))$,	\item[iv)]$\Hann(M):=\prod_{i=0}^d\Ann_R\Ext^i_R(M,R)$.
\end{enumerate}
We need to find situations equipped with nontrivial annihilators  of higher Ext-modules:
\begin{proposition}
	Let $M$ be finitely generated and $L$ be any module. Then each element of $\Ext^i_R(M,L) $ has a nontrivial annihilation for all $i>0$.
\end{proposition}

\begin{proof}
	Suppose on the contradiction that there is  $x\in \Ext^i_R(M,L)$ which is a faithful element. Then
 $R=\frac{R}{(0:x)}   \simeq   xR\hookrightarrow \Ext^i_R(M,L).$  We bring the following claim:
	\begin{enumerate}
		\item[Claim A)] Let $M$ be a finitely generated module over any commutative ring and let $L$ be any module. Then $\Ext^i_R(M,L)$ has no nonzero projective submodule for all $i>0$. \item[ ]Indeed, 	we argue by induction on $i$.  The case $i=1$ is in \cite[Theorem 4.1]{comment}. We make use of the standard shifting
		isomorphisms. Let $E(L)$ be the injective envelope of $L$.
		For the higher ext, its enough to look at $0\to L\to E(L)\to E(L)/L\to 0$ and apply the induction hypothesis throughout  natural identification $\Ext^{i-1}_R(M,E(L)/L)\simeq \Ext^i_R(M,L)$ for all $i>1$.
	\end{enumerate}
In view of  Claim A) we know that $\Ext^i_R(M,L)$ has no nonzero projective submodule.
	This yields a contradiction that we search for it.
\end{proof}

\begin{example}
Let $R:=\mathbb{Q}[[x,y,z]]$  and set $M:=R/(xy,xz)$. One has   $\gamma(M)=\Ann(M)=\Hann(M)=\beta(M)=\alpha(M).$ Also, $\gamma(xy,xz)\neq\Ann(xy,xz)$.
\end{example}

\begin{proof}
In the case of modules of finite projective dimension, the claims $\gamma(M)=\beta(M)=\alpha(M)$
is in \cite{BB}. In view of Example \ref{exa}, $\gamma(M)=\Ann(M)=\Hann(M)$. Again, Example \ref{exa} implies that $\gamma(xy,xz)\neq\Ann(xy,xz)$.
\end{proof}

\begin{observation} 
	If  $M$ is  finitely generated over an integral domain, then $0\neq\gamma(M)$ and $0\neq\beta(M)$.
\end{observation}

\begin{proof}By generic of freeness, there is  $a\in R$ such that $M_{a}$ is free as an $R_a$-module. Let $x$ be a suitable power of $a$. Due to \cite{BB},  $x\Ext^i_R(M,R)=0$ for all $i>0$. In particular, $\gamma(M)$ and $\beta(M)$ are nonzero.
\end{proof}

The finitely generated assumption is important:
With the same notation as of Example 2.1,
	$\Ext^1_R(F/R,R)\simeq R$. Thus, $\gamma(F/R)= 0$.

\begin{proposition}\label{cor}
Let $M$ be a finitely generated module and of positive grade. Then $\gamma(M)=\rad(\Ann(M))$.
\end{proposition}

\begin{proof}The grade condition says that $\Ext^0_R(M,R)=0  (\dag)$.  Set $\fc_i(M)=\rad(\Ann(\Ext^i_R(M,R)))$.
In view of  Proposition \ref{f}  we conclude that

$$
\begin{array}{ll}
\V(\Ann(M))&=\Supp(M) \stackrel{(\dag)}= \bigcup_{i=1}^{d} \Supp(\Ext^i_R(M,R))\\&=\bigcup_{i=1}^{d} \V(\fc_i(M)) =\V(\bigcap _{i=1}^{d} \fc_i(M) ).
\end{array}$$
Since $\Ann(M)\subset\bigcap _{i=1}^{d} \fc_i(M)$, we get from the displayed item that $\bigcap _{i=1}^{d} \fc_i(M)=\rad(\Ann(M)).$
 Again we look at $\rad(\Ann(M))\subset  \fc_i(M)$. Immediately, we deduce that $$\rad(\Ann(M))\subset\bigcap _{i=1}^{\infty} \fc_i(M)\subset\bigcap _{i=1}^{d} \fc_i(M)=\rad(\Ann(M)).$$Consequently, $\gamma(M)=\bigcap _{i=1}^{\infty} \fc_i(M)=\rad(\Ann(M))$.
\end{proof}

\begin{example} The first item shows that finitely generated assumption  is important.
	The second item shows that positivity of grade  is important.
\begin{enumerate}
	\item[i)] Adopt the notation of Example 2.1. Note that $\Hom_R(F,R)=0$. Thus $F$ is of positive grade
and $\gamma(F)=R$. It remains to note that $\Ann(F)=0$.
	\item[ii)]  Let  $(R,\fm)$ be a  Cohen-Macaulay local domain of dimension $d>0$
	with isolated  Gorenstein singulaity and possessing a canonical module.  By Proposition \ref{omega}
$\Ext^+_R(\omega_R,R) $ are of finite length and there is an $i>0$ such that $\Ext^i_R(\omega_R,R)\neq0 $. From this $\gamma(\omega_R)=\bigcap _{i>0}\rad(\Ann_R\Ext^i_R(\omega_R,R))=\fm$. However, $\omega_R$
is faithful. Thus, 	$\gamma(\omega_R)=\fm\neq 0=\rad(\Ann(M))$.
\end{enumerate}
\end{example}

\begin{corollary}\label{corp}
Let $M$ and $N$ be modules of positive grade with the same support. Then $\gamma(M)=\gamma(N)$.
\end{corollary}

\begin{example} The  positivity of grade  is important.
  Let  $(R,\fm)$ be a  Cohen-Macaulay local domain of dimension $d>0$
		with isolated  Gorenstein singulaity and possessing a canonical module. Note that $\omega_R$ and $R$
		have a same support. By the previous example, $\gamma(\omega_R)=\fm\neq R=\gamma(R)$.
\end{example}
\begin{corollary}\label{corp1}
Let $M$  be a finitely generated module of finite projective dimension. Then $\gamma(M)=\rad(\Ann(M))$ if and only if $M$ is of positive grade.
\end{corollary}

\begin{proof}One direction is in Proposition \ref{cor}, without any assumption on the projective dimension. To see the other side implication, we may assume that $\gamma(M)\neq\rad(\Ann(M))$. Since  $\rad(\Ann(M)) \subsetneqq \gamma(M)$, there is an $M$-regular element in $\gamma(M)$.
In view of \cite[Theorem 7.57]{brid},
 $M$ is a 1-syzygy module. It turns out that $M$ is subset of a free module. Therefore, $\Hom_R(M,R)\neq 0$. So, $\grade(M)=0$.
\end{proof}

\begin{corollary}\label{corb}
Let $M$ and $N$ be modules of positive grade and of finite projective dimension with the same support.  Then $\beta(M)=\beta(N)$.
\end{corollary}

\begin{proof}
In view of \cite[Proposition 4]{brid}, $\beta(M)=\gamma(M)$. The claim follows  by support sensitivity of $\gamma(-)$.
\end{proof}

Over polynomial rings, right hand side of ii)  is in  \cite[Page 215]{e}:

 \begin{lemma}
 	Let $R$ be local and $M$ be finitely generated. The following assertions hold:
 	\begin{enumerate}
 		\item[i)]   $\Hann(M)$ and $\Ann(M)$ have  same radical.
 		\item[ii)]  
  If $R$ is Gorenstein, then  $\Ann(M)^{\Gdim (M)-g+1}\subseteq\Hann(M)\subseteq\Ann(M).$ 
\end{enumerate}
\end{lemma}

\begin{proof}Let $d:=\dim R$ and $G:=\Gdim (M)$.

i)  Recall that $\Hann(M):=\prod_{i=0}^d\Ann_R\Ext^i_R(M,R).$   By Proposition \ref{f}, $$ \V(\Ann(M))=\Supp(M)=\bigcup_{i=0}^{\infty} \Supp(\Ext^i_R(M,R))=\bigcup_{i=0}^{d} \Supp(\Ext^i_R(M,R)).$$
Thus,	$\Hann(M)$ and $\Ann(M)$ have  same radical.
	
 ii)   Recall from  \cite[Page 350]{ss} that $\prod_{i=0}^d\Ann_RH^i_{\fm}(M)\subset\Ann(M)$.  By  local duality, $\Hann(M)\subseteq\Ann(M).$ Recall that $g=\inf\{i:\Ext^i_R(M,R)\neq 0 \}$ and that $\Gdim (M)=\sup\{i:\Ext^i_R(M,R)\neq 0 \}$. By  local duality,
$
\begin{array}{ll}
\Ann(M)^{G-g+1} \subset\prod_{i=g}^{G}\Ann_R\Ext^i_R(M,R) =  \prod_{i=0}^d\Ann_R\Ext^i_R(M,R)  =\Hann(M).
\end{array}$
\end{proof}

 \begin{corollary}
Let $R$ be Gorenstein and $M$ be quasi-perfect. Then  $\Hann(M)=\Ann(M)$.
 \end{corollary}

\begin{proof}
	Recall that $g$ is the grade  of $M$. The assumptions grantee  that $\Gdim (M)-g+1=1$. Now the desired claim follows by part ii) of the above lemma.
\end{proof}

The quasi-perfect assumption is important:

\begin{example} 
Suppose $(R,\fm)$ is a polynomial ring over an infinite field and of dimension $d>3$. Let $2\leq i_1<\ldots< i_{\ell}\leq d-2$. Evans
and  Griffith \cite[Theorem A]{eg} constructed a   prime ideal $\fp$ which is  not maximal such
that $\Ext^{d-i+1}_R(R/\fp,R)\simeq R/\fm$ for $i\in\{i_1,\ldots, i_{\ell}\}$ and zero exactly when $i\notin\{i_1,\ldots, i_{\ell},d-2\}$.  The undetermined Ext is $\Ext^2_R(R/\fp,R)$. Put $\fa_2:=\Ann(\Ext^{2}_R(R/\fp,R))$.
By Observation \ref{obd}, $\dim(R/\fa_2)=d-2$. As $\fp$ is a $(d-2)$-dimensional prime ideal and $\fp\subseteq\fa_2$,  we have $\fa_2=\fp$.
Thus,
$\Hann(R/\fp)=\fm^{\ell}\fp\neq \fp=\Ann(R/\fp)$.
\end{example}

\section{Comments on $\eass(-)$}

It may be $\eass(-)=\emptyset$ for a nonzero module $(-)$. To see this,
let $R$ be a  complete local integral domain of positive dimension. Let $F$  be the fraction field of $R$. Recall from Example \ref{commentcomplete} that 
$\Ext^i_R(F,R)=0$ for all $i$. So, $\eass(F)=\emptyset.$
However, we have:

\begin{observation}
	Let  $A$ be a noetherian ring	and $M\neq 0$ be finitely generated. Then $\eass(M)\neq\emptyset$.
\end{observation}

\begin{proof} Since $M$ is finitely generated, $S^{-1}\Ext^i_A(M ,A)=\Ext^i_{A_S}(M_S,A_S)$. Also,
	$\fp\in\Ass_A(-)$ if and only if $\fp A_{\fp}\in\Ass_{A_{\fp}}(-_{\fp})$. From these,
	$\fp\in\eass(M)$ if and only if $\fp A_{\fp}\in\eass(M_{\fp})$.
Then, without loss of the generality we may assume that $A$ is local. By Lemma \ref{fg},
there is an $i$ such that
	$\Ext^i_A(M,A)\neq 0$.  Consequently,
	$\emptyset\neq\Ass(\Ext^i_A(M,A))\subset\eass(M)$.
\end{proof}

The same argument shows:
\begin{corollary}
	Let  $A$ be a noetherian ring	and $M $ be homlogically nonzero. Then $\eass(M)\neq\emptyset$.
\end{corollary}

Over a Gorenstein local  ring and for a  Cohen-Macaulay module $M$
we know that $\eass(M)=\hass(M)=\Ass(M)$. The Gorenstein assumption is important:

\begin{proposition}\label{omega}
Let  $(R,\fm)$ be a  Cohen-Macaulay local ring of dimension $d>0$
with isolated  Gorenstein singulaity and possessing a canonical module.
Then  $\Ass_R(\omega_R)=\min(R)\subsetneqq\min(R)\cup\{\fm\}=\eass(\omega_R).$\end{proposition}

 \begin{proof}
 By isolated Gorenstein singularity we mean a non Gorenstein ring which is Gorenstein over the punctured spectrum.	
	Clearly, $\min(R)=\Ass_R(\omega_R)$.
Recall that $\Ass(\Hom_R(\omega_R,R))=\Supp(\omega_R)\cap\Ass(R)=\Spec(R)\cap\Ass(R)=\min(R)$.
Let $\fp\in\Spec(R)\setminus\{\fm\}$ (such a thing exists, since $d>0$).
Since $R_{\fp}$ is Gorenstein, $(\omega_R)_{\fp}$
is maximal Cohen-Macaulay and in view of Auslander-Bridger formula, $(\omega_R)_{\fp}$  is totally reflexive. We deduce that  $\Ext^+_{R_{\fp}}((\omega_R)_{\fp},R_{\fp})=0$.
This means that $\Supp(\Ext^+_R(\omega_R,R))\subset\{\fm\}$ and consequently, $\Ass(\Ext^+_R(\omega_R,R))\subset\{\fm\}$. Suppose on the contradiction
that $\Ext^i_R(\omega_R,R)=0$ for all $i>0$. We recall that:
\begin{enumerate}
	\item[Fact A)] (See e.g. \cite[2.2]{hh}) Let  $A$ be a generically Gorenstein local ring.
	If $\Ext^i_A(\omega_A,A)=0$  for all $i \in[1, \dim A]$, then $A$ is Gorenstein.
\end{enumerate}
By this fact $R$ is Gorenstein, a contradiction. Thus, there is an $i>0$ such that  
$\Ext^i_R(\omega_R,R)\neq0$. Hence, it is of finite length.
Therefore, $\eass(\omega_R)=\cup_{j}\Ass(\Ext^j_R(\omega_R,R))=\min(R)\cup\{\fm\}$.
\end{proof}

\begin{remark} Suppose $M_1\subset M_2$. The first item shows that it may be $\eass(M_1)\nsubseteqq\eass(M_2)$. The second item shows that  it may be $\eass(M_2/M_1)\nsubseteqq\eass(M_2)$.

i) Adopt the notation of  Proposition \ref{omega}, take $M_1:=\omega_R$ and $M_2:=R$.
	Then $\eass(M_1)\nsubseteqq\eass(M_2)$.

ii)  If $M_1$ is a direct summand of $M_2$, then 
$\eass(M_2)=\eass(M_2/M_1)\cup\eass(M_1).$ 

iii) Let $R$ be of positive depth. We look at $M_1:=\fm\subset M_2:= R$. It remains to  remark that
$\eass(M_2/M_1)=\{\fm\}\nsubseteqq\Ass(R)=\eass(M_2)$. 
\end{remark}

\begin{proposition}
	Let $0\to M_1\to M_2\to M_3\to 0$ be an exact sequence of finitely generated modules and assume that $M_3$ is locally free over the punctured spectrum. Then $\eass(M_2)\subset\eass(M_1)\cup\eass(M_3)$.
\end{proposition}

 \begin{proof}
 We look at  $0\to   M_3 ^\ast\to   M_2 ^\ast\to   M_1  ^\ast\to  \Ext^1_R(M_3,R)\to \Ext^1_R(M_2,R)\to \Ext^1_R(M_1,R)\to\cdots$.
We have $0\to   M_3^\ast\to  M_2^\ast\to X\to 0$ where $X\subset M_1^\ast$. This implies that
$$\Ass(M_2^\ast)\subset\Ass(M_1^\ast)\cup\Ass(X)\subset\Ass(M_2^\ast)\cup\Ass(M_1^\ast)\quad(+)$$
For each $i>0$, we have $  \Ext^i_R(M_3,R)\to \Ext^i_R(M_2,R)\to \Ext^i_R(M_1,R)$. Then 
$0\to Z\to \Ext^i(M_2,R)\to X\to 0$ where $X\subset\Ext^i(M_1,R)$ and 
$\Ext^i(M_3,R)\to Z\to  0$. Since $M_3$ is locally free, it follows that $\ell(\Ext^i(M_3,R))<\infty$.
Then,  $\Ass(\Ext^i_R(M_3,R))\subset\{\fm\}$. The same thing holds for $Z$.
If $Z=0$, then $$\Ass(\Ext^i_R(M_2,R))=\Ass(X)\subset\Ass(X)\cup\Ass(\Ext^i_R(M_3,R))= \Ass(\Ext^i_R(M_1,R))\cup \Ass(\Ext^i_R(M_3,R))$$
If $Z\neq 0$, then  $\Ass(\Ext^i(M_2,R))\subset\Ass(X)\cup\Ass(Z)\subset\Ass(\Ext^i(M_1,R))\cup \Ass(\Ext^i_R(M_3,R))$. In both cases we have $$\Ass(\Ext^i_R(M_2,R))\subset\Ass(X)\cup\Ass(Z)\subset\Ass(\Ext^i_R(M_1,R))\cup \Ass(\Ext^i_R(M_3,R))\quad(+,+)$$ Now we use $(+)$ and $(+,+)$ to deduce that

 $$
\begin{array}{ll}
\eass(M_2)&=\Ass(M_2^\ast)\cup\left(\cup_{i>0}\Ass(\Ext^i_R(M_2,R))\right)\\& \subset\Ass(M_2^\ast)\cup\Ass(M_1^\ast)\cup\left(\cup_{i>0}(\Ass(\Ext^i_R(M_1,R))\cup \Ass(\Ext^i_R(M_3,R))\right)\\&=\eass(M_1)\cup\eass(M_3).
\end{array}$$
\end{proof}

\begin{proposition}
	Let $R$ be local and $M$ be finitely generated. Then $\min(M)\subset\eass(M)$.
	In particular, $\Ass(M)\subset\eass(M)$ provided $M$ is equi-dimensional.
\end{proposition}

\begin{proof}
	Let $\fp\in\min(\Supp(M))$.
	In view of Proposition \ref{f} we see $\fp\in\min(\eSupp(M))$. By definition,
	there is an $i$ such that $\fp\in\Supp(\Ext^i_R(M,R))$.
	Let $\fq\subset\fp$ be in $\Supp(\Ext^i_R(M,R))$. Then
	$\fq\in\Supp(M)$. Since $\fp$ is minimal in support of $M$, we deduce that
	$\fq=\fp$, i.e., $\fp\in\min(\Supp(\Ext^i_R(M,R)))$. Recall that minimal element
	of support of a finitely generated module is an associated prime ideal. From this,
	$\fp\in\Ass(\Ext^i_R(M,R))\subset\eass(M)$.
\end{proof}

\begin{proposition}
	Let $R$ be reduced and $M$ be  finitely generated. Then $\Ass(R)\cap\eass(M)=\Ass(R)\cap\eSupp(M).$
\end{proposition}

\begin{proof}
	Suppose there is  $\fp\in \Ass(R)\cap\eass(M)$. Then $\fp=(0:r)$ for some $r\in R$ and $R/\fp$ embedded into $\Ext^i_R(M,R)$ for some $i\geq 0$. We claim that $i=0$. Suppose on the contradiction that $i>0$. Then $\Ext^i_R(M,R) =\Ext^1_R(\Omega^{i-1}(M),R)$.  Here $\Omega^{i-1}(M)$ is the $(i-1)$-th syzygy module of $M$, this is well-defined because $i\neq 0$. 
 We recall from  \cite[Corollary 4.2]{comment} that:
	\begin{enumerate}
		\item[Fact A)] If $N\subset\Ext^1_R(M,-)$, then $\tr(N)$ is nil.
	\end{enumerate}
	In the light of Fact A)  we see trace of $R/\fp$ is nilpotent. As $R$ is reduced, $\tr(R/\fp)=0$. But $0\neq r\in \tr(R/\fp)$. This   contradiction
	says that $i=0$. Since $$\Ass(\Ext^0_R(M,R))=\Ass(R)\cap\Supp(M)\stackrel{\ref{f}}=\Ass(R)\cap\eSupp(M),$$ we get the claim.
\end{proof}

\begin{proposition}\label{ind}
Let $R$ be a  $d$-dimensional local ring and suppose Problem 1.1 is true over $d-1$ dimensional rings. Then $|\eass(I)|<\infty$  for all ideal $I\lhd R$ of positive grade.
\end{proposition}

\begin{proof}We look at the induced long exact sequence of Ext-modules induced from  
	$0\to I\to R\to R/I \to 0 $
to see that $\Ext^i_R(I,-)\simeq\Ext^{i+1}_R(R/I,-)$ for all $i>0$.
In particular,
$I\Ext^{>0}_R(I,R)=0$. Let $y\in I$ be a regular element.
Look at the exact sequence  $ 
0\to I \stackrel{y}\lo I \to  I/yI \to 0.$
This induces the exact sequence $$\begin{CD}
0@>>> \Ext^i_R(I,R) @>
>> \Ext^{i+1}_R(I/yI,R) @>>> \Ext^{i+1}_R(I,R)  @>>> 0.
\\
\end{CD}
$$
Thus $|\eass_R(I)|<\infty$ provided $|\eass_R(I/yI)|<\infty$. By \cite[Page 140]{mat}, there is the isomorphism $\Ext^{i+1}_R(I/yI,R)\simeq \Ext^i_{\overline{R}}(I/yI,\overline{R}),$ we observe that $\eass_R(I)$ is finite if  $\eass_{\overline{R}}(I/yI)$ is finite.
\end{proof}

\begin{proposition}\label{proj}The following holds:\begin{enumerate}
		\item[i)] Problem 1.1 reduces to complete case.
	\item[ii)] Problem 1.1 reduces to maximal  Cohen-Macaulay modules, when $R$ is Cohen-Macaulay. \item[iii)] Problem  1.1 is true if $R$ is Gorenstein over the punctured spectrum.
	\item[iv)]Problem 1.1 may reduce to cyclic modules:
	Suppose $R$  is normal and Cohen-Macaulay. The following are equivalent:
	\begin{enumerate}
		\item[a)] $|\eass(M)|<\infty$ for all finitely generated $M$.
		\item[b)] $|\eass(\fb)|<\infty$  for all ideal $\fb$
		of  max-height at most two. \item[c)]  $|\eass(R/\fb)|<\infty$ for all ideal $\fb$
		of max-height at most two.\end{enumerate}
\end{enumerate}
\end{proposition}

\begin{proof}i) Let $\fp$ be a prime ideal and let $\fq\in \Ass(\widehat{R}/\fp\widehat{R})$.  Indeed, if $x\in R\setminus \fp$, then $R/ \fp \stackrel{x}\to R/ \fp$
is injective and so $\widehat{R}/ \fp\widehat{R} \stackrel{x}\to \widehat{R}/ \fp\widehat{R}$ is injective. Thus, $x\notin\bigcup _{Q\in \Ass(\widehat{R}/\fp\widehat{R})}Q$.
Therefore, $\fq\cap R=\fp$. This show that if $\fp_1\neq \fp_2$ then $\Ass(\widehat{R}/\fp_1\widehat{R})\neq \Ass(\widehat{R}/\fp_2\widehat{R})\quad(\ast)$.
Now suppose that $\bigcup_i\Ass_{\widehat{R}}(\Ext^i_R(M\otimes\widehat{R},\widehat{R})$ is finite. In view of
\cite[Theorem 23.2]{mat},
$ \Ass_{\widehat{R}} (\Ext^i_R(M,R)\otimes_R\widehat{R})=\bigcup _{\fp\in\Ass\Ext^i_R(M,R)}\Ass(\widehat{R}/\fp\widehat{R}).$ Combine this along with $(\ast)$
we deduce that $\bigcup_i\Ass(\Ext^i_R(M,R))$ is finite.

ii)
Let $d:=\dim R$.   We may assume that $M$ is nonzero and that $\pd(M)=\infty$.
We look at the following exact
sequence:
 $0\to\Omega_{d}\to F_{d-1}\to\ldots\to F_{0}\lo M\to 0$ 
where $F_{i}$ is finitely generated free for all $i=0,\ldots,d$. Then $\Omega_{d}\neq 0$.
In view of \cite[Exercise 2.1.26]{BH} $\Omega_{d}$ is
maximal Cohen-Macaulay. Also, $\Ext^{i+d}_R(M,R)\simeq\Ext^i_R(\Omega_{d},R)$.  It turns out that finiteness of $\eass_R(\Omega_{n})$ implies the finiteness of $\eass_R(M)$. Without loss of the generality, we may assume that $M$ is maximal  Cohen-Macaulay.

iii)  Let $i\geq\dim R=:d$.
Let  $\fp\neq \fm$ be a prime ideal. Since $R_{\fp}$ is Gorenstein, $\id(R_{\fp})=\Ht(\fp)<\infty$. Thus,
$\Ext^i_{R_{\fp}}(M_{\fp},R_{\fp})=0$. Consequently, $\Supp(\Ext^i_R(M,R))\subset\{\fm\}$.
Then   $\eass(M)\subset \bigcup_{i=0}^{d-1}\Ass(\Ext^i_R(M,R))\cup \{\fm\}.$ 
So, $\eass(M)$ is  finite.

iv) Recall that  $\fa$ has \textit{max-height} at most $t$ if  $\Ht(\fp)\leq t$ for all $\fp\in\Ass(R/\fa)$. For simplicity, we recall from \cite[Theorem B.b]{aus} that:
	\begin{enumerate}
	\item[Fact A)] 	Let $R$ be an integrally closed Cohen-Macaulay ring and $N$  a finitely generated $R$-module. There is an ideal $\fb$ of max-height at most 2
	such that
	$\Ext^i_R(\fb,-)\simeq \Ext^{i+2}_R(N,- )$  $\forall i\geq2$.
\end{enumerate}
Recall  that $\Ext^i_R(\fb,-)\simeq\Ext^{i+1}_R(R/\fb,-)$  for all $i>0$. By Auslander's remark  and for all $i\geq 2$, we have $\Ext^{i+1}_R(R/\fb,-)\simeq \Ext^{i+2}_R(M,- )$. We note that the initial terms $\Ext^{i}_R(M,- )$
are not effective, because the associated prime ideals of a finitely generated module is a finite set.
This yields the claim.
\end{proof}

\begin{proposition}\label{three}Adopt one of the following assumptions:
	\begin{enumerate}
		\item[i)] $	R$ is a 3-dimensional excellent normal local domains.
\item[ii)] $	R$ is a  2-dimensional reduced excellent local rings.
\end{enumerate}Then Problem 1.1 is true.
\end{proposition}

\begin{proof} Let us bring a general phenomena. The  Gorenstein locus $\Gor(X)$ of $X:=\Spec(R)$ is the set of all primes $\fp$ such that $ R_{\fp}$ is  Gorenstein.
Since $R$ is excellent and in view of \cite{GOR}, there is an ideal $I$ such that
$X\setminus \Gor(X)=\V(I)$.

 i)  Normal rings are $(R_1)$.
As $R$ is normal,  $\Ht(I)>1$. Due to $\dim R=3$  we get that $\dim R/I\leq 1$. Thus, for all $i\geq3$ we have $\Supp(\Ext^i_R(M,R))\subset\V(I)$.
Therefore,  $ \eass(M)\subset \bigcup_{i=0}^{2}\Ass(\Ext^i_R(M,R))\cup \V(I).$ 
Since $\V(I)$ is a finite set, it follows that $\eass(M)$ is  finite.

 ii)   Reduced rings are  $(R_0)$. This implies  that $\Ht(I)>0$.
Similar to  part i),  $\eass(M)$ is  finite.
\end{proof}

\begin{lemma}\label{22}
Let $R$ be a 2-dimensional local ring. Suppose $\eass(I)$ is finite for all unmixed ideals of codimension zero. Then $\eass(I)$ is finite for all $I$.
\end{lemma}

\begin{proof}Let $X$ be a module of dimension at most one. Then its support is finite. 
We may assume that $\Ht(I)=0$.
Suppose $I$ is not unmixed. We write $I=I_1 \cap I_2$, where $I_1$ is the unmixed part and $I_2$ is the mixed part.  As $\Ht(I_2)>0$,  $|\eass(R/I_2)|\leq|\Supp(R/I_2)|<\infty$.
Similarly,  $|\eass(R/I_1+I_2)|<\infty$.
We look at the exact sequence  $0\to R/I \to R/I_1\oplus R/I_2\to R/(I_1 + I_2) \to 0.$ This induces the following long exact sequence of Ext-modules$$\ldots\to\Ext^{i-1}_R(R/(I_1 + I_2),R)\stackrel{\rho_i}\lo \Ext^{i}_R(R/I,R)\stackrel{\sigma_i}\lo\Ext^{i}_R(R/I_1,R)\oplus\Ext^{i}_R(R/I_2,R)\to \ldots $$
Set  $E_i:=\frac{\Ext^{i-1}_R(R/(I_1 + I_2),R)}{\ker(\rho_i)}.$ In particular, $\bigcup_i \Supp(E_i)$ is finite.
 Set $D_i:=\im(\sigma_i)$. Then
$\bigcup_i \Ass(D_i)$ is finite. Now look at the exact sequence  $0\to E_i \to \Ext^{i}_R(R/I,R)\to D_i \to 0.$  This implies that
$\Ass(\Ext^{i}_R(R/I,R))\subseteq\Ass(E_i)\cup\Ass(D_i).$ 
\end{proof}

\begin{example}
Let $R$ be either $\mathbb{F}_p[[X,Y,Z]]/(X^p+Y^p+Z^p)$ or  
$k[[X,Y,Z]]/(X^2,XYZ)$.
Then $\eass(I)$ is finite for all $I$.
\end{example}

\begin{proof} First, let $R:=\mathbb{F}_p[[X,Y,Z]]/(X^p+Y^p+Z^p)$. In view of the above lemma, we assume $I$ is unmixed and of codimension zero.  Note that $R:=\mathbb{F}_p[[X,Y,Z]]/(X+Y+Z)^p$.
Set $\xi:=(x+y+z)^i$  and $\zeta:=(x+y+z)^{p-i}$. Without loss of the generality  we assume that $I=\xi R$. Its free resolution is given by 
$\cdots \lo  R  \stackrel{\xi}\lo  R  \stackrel{\zeta} \lo R\stackrel{\xi}\lo I.$ We deduce  that $|\eass(I)|<\infty$.
Now, let $R:=k[[X,Y,Z]]/(X^2,XYZ)$. 
In view of the above lemma, we may assume that $I=(x)$. Set $a:=yz$. Note that $R_a\simeq k[[y,z,y^{-1},z^{-1}]]$ which is regular. It turns out that $\Ext^i_R(I,R)$ is annihilated by some uniform power of $a$   for all
$i\geq2$, see  Proposition \ref{proj}(iii). Thus $\Supp(\Ext^i_R(I,R))\subset \V(I+yz R)$ for all
$i\geq2$. Note that $\V(I+yz R)$
is finite. So,
$\eass_R(I)\subseteq\bigcup_{i=0}^1\Ass(\Ext^i_R(I,R))\cup\V(I+yz R),$  which is a finite set.
\end{proof}

\begin{remark}
\label{red1}
Let $R$ be a normal closed Cohen-Macaulay local domain of dimension  four. Suppose
$|\eass(R/\fb)|<\infty$ for  almost complete-intersection ideal $\fb$ of height 1.
The following are equivalent:
\begin{enumerate}
\item[i)] $|\eass(R/\fb)|<\infty$ for all ideal $\fb$
of height exactly two.
\item[ii)]   $|\eass(M)|<\infty$ for all finitely generated $M$.
\item[iii)] $|\eass(R/\fb)|<\infty$ for all unmixed ideal $\fb$
of height exactly two.
\end{enumerate}
\end{remark}

\begin{proof}
$i)\Longrightarrow ii)$:
Keep the above lemma in mind and let $\fb$
be an ideal of max-height at most two.  If $\Ht(\fb)=2$ we are done by the assumption.
Suppose
$\Ht(\fb)=1$. We claim that things reduce to the unmixed case.
To this end, we write $\fb=\fb_1 \cap \fb_2$, where $\fb_1$ is the unmixed part and $\fb_2$ is the mixed part. We assume that  the claim is true for $\fb_1$. If $\Ht(\fb_2)=2$, then by the assumption $|\eass(R/\fb_2)|<\infty$. Thus, $\Ht(\fb_2)>2$. Consequently, $\dim R/\fb_2\leq 1$. Here, we used our low-dimensional assumption. This yields that $\eass(R/\fb_2)\subset\V(\fb_2)$ which is a finite set.
By the same reasoning, the claim is true for the ideal $\fb_1+\fb_2$, because $\Ht(\fb_1+\fb_2)>1$.
We look at the exact sequence  $0\to R/\fb \to R/\fb_1\oplus R/\fb_2\to R/(\fb_1 + \fb_2) \to 0.$ This induces the following long exact sequence of Ext-modules$$\ldots\to\Ext^{i-1}_R(R/(\fb_1 + \fb_2),R)\stackrel{\rho_i}\lo \Ext^{i}_R(R/\fb,R)\stackrel{\sigma_i}\lo\Ext^{i}_R(R/\fb_1,R)\oplus\Ext^{i}_R(R/\fb_2,R)\to \ldots $$
The  reasoning given by Lemma \ref{22} allow us to assume that $\fb$ is unmixed and is of height one.
In view of  \ref{proj}, we may assume that $\fb$ is not principal. Let $b$ such that it generates $\fb$ at the minimal
primes of $\fb$.
Look at the irredundant  decomposition of $(b)=\fb\cap \fc$. Let $c$ be in $\fc$ but not in minimal primes of $\fb$. If $\Ht(b,c)=2$, then  $(b,c)$ is complete-intersection and the claim follows by \ref{proj}(iii). So, $(b,c)$ is almost complete intersection.

Claim A. One has $\fb=(b:c)$. Indeed, let $x\in \fb$. Then $xc\in\fb \cap \fc=(b)$. For the other side inclusion,  let  $x\in(b:c)$. Then
$xc\in(b)=\bigcap q_i\cap \bigcap \widetilde{q}_j$ where $q_i$ (resp. $\widetilde{q}_j$) are primary components of $\fb$ (resp. $\fc$).
Since$\fb$ is unmixed,  $\{\rad(\fq_i)\}_i=\min(\fb)$. Due to the choose of $c$, $c\notin\rad(\fq_i)$ for all $i$. Since $xc\in \fq_i$ and $c\notin\rad(\fq_i)$,
 we deduce from the definition of primary ideals  that $x\in\fq_i$ for all $i$, i.e., $x\in(b)$. So, $(b:c)\subset \fb$. This proves the claim.

Claim A) induces the following exact sequence
 $
0\to R/\fb \stackrel{c}\lo
  R/(b) \to  R/(b,c)\to 0.
$
Keep in mind that
$\pd(R/bR)=1$. The induced long exact sequence tells us $\Ext^{i}_R(R/\fb,-)\simeq \Ext^{i+1}_R(R/(c,b),-)$ for all $i>1.$ This completes the proof of the implication  $i)\Longrightarrow ii)$.

$ii)\Longrightarrow iii)$: This is trivial.

$iii)\Longrightarrow i)$:   The proof is a repetition of the above argument. We left it to the reader.
\end{proof}

\section{Negative answer to Problem 1.1}

Let $S:=\Q[[s,t]]$ and define
$B:= S[x_1,x_2,x_3,x_4]/
	(x_1^2,x_2^2,x_3^2,x_4^2,x_3x_4,
	2x_1x_3+x_2x_3,x_1x_4+x_2x_4).$

\begin{remark}\label{r62}The symbols $x_j$ also denote their residue classes in $B$.
With $a=x_1x_2$, $b=x_1x_3$, and $c=x_1x_4$, an $S$-basis of $B$ is
$\{ 	1,x_1,x_2,x_3,x_4,a,b,c\}$.
Indeed, in degree two the seven displayed defining quadrics are
independent in the ten-dimensional space of quadrics, and all monomials
of degree at least three vanish. In particular,
$x_2x_3=-2b$ and $ x_2x_4=-c.
$
Let $N$ be the free $S$-module with basis $u,v,w,z$. Make it a $B$-module
by declaring $x_jw=x_jz=0$ for $1\leq j\leq4$ and using this table:
$$
(\ast):\quad \quad	\begin{array}{c|cc}
		&u&v\\ \hline
		x_1&w&sz\\
		x_2&w&z\\
		x_3&z&w\\
		x_4&tz&w
	\end{array}
$$
Every product $x_jx_k$ acts as zero on $N$, so all the relations
of $B$ hold on $N$. Thus this is a well-defined $B$-module.
\end{remark}
\begin{definition}Define the idealization
$
R:=B\ltimes N$
where 
$(b,n)(b',n')=(bb',bn'+b'n).$
\end{definition}
We regard $B$ as a subring of $R$ and $N$ as a square-zero ideal. Set
$J:=(x_1,x_2,x_3,x_4)B\oplus N\ \subseteq R.$ Every product \(x_jx_k\) acts as zero on \(N\), so all the relations
of \(B\) hold on \(N\). Hence \(B_2\cdot N=0\), so \(I^2N=0\).
Then $J^3\subseteq I^3\oplus I^2N=0\oplus0=0$ and $R/J\cong S$. Consequently $J$ is the nilradical,
$R$ is local with maximal ideal $(s,t)R+J$, and $\dim R=2$.
The ring $R$ is finite free of rank $12$ over $S$, hence Noetherian
and Cohen--Macaulay. It is complete: the $(s,t)R$-adic and maximal-ideal-adic
topologies agree because $J$ is nilpotent, and a finite free $S$-module
is complete.
For $q\in\Q^\times$, write
\(
d(q)=\begin{pmatrix}x_1&qx_3\\x_4&x_2\end{pmatrix} \) where
\(d_i=d(2^i)\) with \(i\in\mathbb Z\). 
\begin{definition}\label{eq:M}
	$M=\coker\left(d_1:R^2\longrightarrow R^2\right).$
\end{definition}

\begin{theorem}\label{thm:main}
	For the above ring and module, let
	$\mathfrak p_i=J+(s-2^it)R$ for $i\geq1$. Then
	\(
	\Ext^i_R(M,R)\cong R/\mathfrak p_i	\) for \(i\geq1\).
	The ideals $\mathfrak p_i$ are pairwise distinct prime ideals, and
	\(
\eass_R(M)=\{J\}\ \cup\ \{\mathfrak p_i:i\geq1\}.
	\)
	Moreover, $M$ is maximal Cohen--Macaulay and its minimal free
	resolution has rank two in every degree.
\end{theorem}

We first check exactness over $B$, and then on the square-zero summand
$N$. The decomposition $R=B\oplus N$ is preserved by every $d_i$ and
$d_i^{\mathsf T}$ as an $S$-linear map.

	\begin{lemma}\label{lem:B}
		For every $q\in\Q^\times$,
		\(
		\ker( d(q)|_{B^2})=\im (d(2q)|_{B^2}),	\)
	and
		\(	\ker (d(2q)^{\mathsf T}|_{B^2})=\im (d(q)^{\mathsf T}|_{B^2}).
		\)
		The images and cokernels in these equalities are free over $S$.
	\end{lemma}
	
	\begin{proof}
		Give $B$ its grading in the variables $x_1,\ldots,x_4$, with $S$ in
		degree zero. Thus $B=B_0\oplus B_1\oplus B_2$ with ranks $1,4,3$,
		where $B_0=S$, $B_1=Sx_1\oplus Sx_2\oplus Sx_3\oplus Sx_4$, and
		$B_2=Sa\oplus Sb\oplus Sc$ with
	$	a=x_1x_2$, $b=x_1x_3$ and $ c=x_1x_4$.
		Recall the relations are
$
		x_2x_3=-2b,  x_2x_4=-c,  x_3x_4=x_j^2=0.
$
		
		\medskip
		\noindent\textbf{Step 1: $d(q)d(2q)=0$.}
		We verify this degree by degree; naive matrix multiplication is not
		sufficient because the individual entries need not vanish.
		\begin{itemize}
			\item \emph{Degree $0$.} On $B_0^2=S^2$ we have
			$d(2q)(1,0)=(x_1,x_4)$ and $d(2q)(0,1)=(2qx_3,x_2)$. Applying
			$d(q)$ and using $x_1^2=x_3x_4=0$, $x_4x_1=c=-x_2x_4$, and
			$x_3x_2=-2b$, $x_2^2=0$:
			\[
			d(q)(x_1,x_4)=(0,\,c-c)=0,
			\qquad
			d(q)(2qx_3,x_2)=(2qb-2qb,\,0)=0.
			\]
			\item \emph{Degree $1$.} The map $d(2q)$ sends $B_1^2\to B_2^2$,
			while $d(q)$ sends $B_2^2\to B_3^2=0$. Hence the composite is zero.
			\item \emph{Degree $2$.} The map $d(2q)$ sends $B_2^2\to B_3^2=0$,
			so the composite is zero.
		\end{itemize}
		Thus $d(q)d(2q)=0$ as a map $B^2\to B^2$; transposing gives
		$d(2q)^{\mathsf T}d(q)^{\mathsf T}=0$.
		
		\medskip
		\noindent\textbf{Step 2}: injectivity on $B_0^2$ and vanishing on $B_2^2$.
		Both $d(q)$ and $d(q)^{\mathsf T}$ send $B_0^2$ to
		$B_1^2$ via the two independent vectors $(x_1,x_4)$ and $(qx_3,x_2)$
		(resp.\ their transposes), so they are injective on $B_0^2$ since
		these are nonzero and $S$-linearly independent in $B_1^2$. Both vanish
		on $B_2^2$ because $B_3=0$.
		
		\medskip
		\noindent\textbf{Step 3}: the degree-one map for $d(q)$ is surjective
			onto $B_2^2$.
		Write $\xi=\sum_j A_jx_j$ and $\eta=\sum_j C_jx_j$ with $A_j,C_j\in S$.
		Using $x_3x_1=-b$, $x_3x_2=-2b$, $x_3x_4=0$, $x_4x_1=c$, $x_4x_2=c$,
		$x_2x_1=a$, $x_2x_3=-2b$, $x_2x_4=-c$, a direct computation gives
		\begin{align*}
			d(q)(\xi,\eta)=\bigl(A_2a+(A_3+qC_1-2qC_2)b+A_4c, 
			 C_1a-2C_3b+(A_1+A_2-C_4)c\bigr).
		\end{align*}
		This map $B_1^2\to B_2^2$ is surjective: preimages of the six basis
		elements are
		\[
		\begin{array}{c|c}
			(a,0)\mapsto(A_2=1) & (0,a)\mapsto(C_1=1)\\
			(b,0)\mapsto(A_3=1) & (0,b)\mapsto(C_3=-\tfrac12)\\
			(c,0)\mapsto(A_4=1) & (0,c)\mapsto(A_1=1),
		\end{array}
		\]
		all other coefficients zero. Hence $d(q)|_{B_1^2}$ is surjective onto
		$B_2^2$.
		
		\medskip
		\noindent\textbf{Step 4}: kernel of $d(q)$ on $B_1^2$.
		Since $a,b,c$ are $S$-linearly independent in $B_2$, the equation
		$d(q)(\xi,\eta)=0$ forces all six coefficients above to vanish:
		\[
		A_2=0,\quad A_3+qC_1-2qC_2=0,\quad A_4=0,\quad 
		C_1=0,\quad C_3=0,\quad A_1+A_2-C_4=0.
		\]
		Thus $A_2=C_1=C_3=A_4=0$ and $A_1=C_4$, while the middle equation
		gives $A_3=2qC_2$. Therefore
		\(
		(\xi,\eta)=A_1(x_1,x_4)+C_2(2qx_3,x_2),
		\)
		so
	\(
		\ker(  d(q)\big|_{B_1^2})=S(x_1,x_4)+S(2qx_3,x_2).
		\)
		This is exactly the image of $d(2q)$ on $B_0^2$, since
		$d(2q)(1,0)=(x_1,x_4)$ and $d(2q)(0,1)=(2qx_3,x_2)$. This proves the
		first assertion.
		
		\medskip
		\noindent\textbf{Step 5}: the transpose.
		The transpose $d(q)^{\mathsf T}=\begin{pmatrix}x_1&x_4\\qx_3&x_2\end{pmatrix}$
		gives, by the same computation,
		\begin{align*}
			d(q)^{\mathsf T}(\xi,\eta)=\bigl(A_2a+A_3b+(A_1+A_2+A_4)c, C_1a+(-qC_1-2qC_2-2C_3)b-C_4c\bigr).
		\end{align*}
		This is again surjective onto $B_2^2$ (the same six preimages work,
		with the obvious adjustments). Solving for the kernel yields
	$
		\ker(  d(q)^{\mathsf T}\big|_{B_1^2})
		=S(x_1,\tfrac{q}{2}x_3)+S(x_4,x_2).
	$
		Replacing $q$ by $2q$ identifies this kernel with the degree-zero
		image of $d(q)^{\mathsf T}$, proving the second assertion.
		
		\medskip
		\noindent\textbf{Step 6}: freeness.
		The kernel $S(x_1,x_4)+S(2qx_3,x_2)$ is a direct summand of $B_1^2$:
		in the basis $(x_1,x_2,x_3,x_4)$ of each copy of $B_1$, the two
		generators have coordinate vectors $(1,0,0,0;0,0,0,1)$ and
		$(0,0,2q,0;0,1,0,0)$, which contain a $2\times2$ identity submatrix.
		Hence it is free of rank $2$, and the cokernel of the inclusion into
		$B_1^2$ is free of rank $8-2=6$. Combined with the surjection onto
		$B_2^2$ (free of rank $3$, kernel free of rank $2$ in degree zero and
		$3$ in degree one) and the vanishing on $B_2^2$, the images and
		cokernels are free over $S$. Each image has $S$-rank $2+3+0=5$, and
		the corresponding cokernel has $S$-rank $16-5=11$; the free summands
		appearing in the acyclicity statement have rank eight.
	\end{proof}
	\begin{proof}[Proof of Theorem~\ref{thm:main}]Put $U=Su\oplus Sv$ and $W=Sw\oplus Sz$, so $N=U\oplus W$.
Every $x_j$ maps $U$ into $W$ and kills $W$.
Using the ordered bases $(u,v)$ and $(w,z)$, its matrix $U\to W$ is
\[
X_1=\begin{pmatrix}1&0\\0&s\end{pmatrix},\quad
X_2=\begin{pmatrix}1&0\\0&1\end{pmatrix},\quad
X_3=\begin{pmatrix}0&1\\1&0\end{pmatrix},\quad
X_4=\begin{pmatrix}0&1\\t&0\end{pmatrix}.
\]
The map induced by $d(q)$ from $U^2$ to $W^2$ is therefore
$$
	D(q)=\begin{pmatrix}X_1&qX_3\\X_4&X_2\end{pmatrix}
	=\begin{pmatrix}
		1&0&0&q\\0&s&q&0\\0&1&1&0\\t&0&0&1
	\end{pmatrix}.
$$
We apply the simultaneous row and column permutation $\sigma=(1,4,2,3)$,
meaning the new index order is
$(1',2',3',4')=(1,4,2,3).$
Let $P$ be the permutation matrix whose rows are the standard basis
vectors in this order:
\[
P=\begin{pmatrix}
	1&0&0&0\\
	0&0&0&1\\
	0&1&0&0\\
	0&0&1&0
\end{pmatrix}.
\]
Thus row $1'$ is $e_1$, row $2'$ is $e_4$, row $3'$ is $e_2$, and row
$4'$ is $e_3$. Since $P$ is a permutation matrix, $P^{-1}=P^{\mathsf T}$.
Left multiplication by $P$ permutes the rows; right multiplication by
$P^{-1}=P^{\mathsf T}$ permutes the columns in the same way. The
conjugation
\(
P\,D(q)\,P^{-1}=P\,D(q)\,P^{\mathsf T}
\)
is a similarity transformation, so it preserves the determinant, the
rank, and the isomorphism type of the cokernel.
 Reordering the rows of $D(q)$ into
the order (row $1$, row $4$, row $2$, row $3$) gives
\[
P D(q)=\begin{pmatrix}
	1&0&0&q\\
	t&0&0&1\\
	0&s&q&0\\
	0&1&1&0
\end{pmatrix}.
\]
Reordering the columns into the
order (col $1$, col $4$, col $2$, col $3$) gives
\[
P D(q) P^{\mathsf T}
=
\begin{pmatrix}
	1&q&0&0\\
	t&1&0&0\\
	0&0&s&q\\
	0&0&1&1
\end{pmatrix}
=
\begin{pmatrix}1&q\\t&1\end{pmatrix}
\ \oplus\
\begin{pmatrix}s&q\\1&1\end{pmatrix}.
\]
 Since the determinant is multiplicative
across the block decomposition,
\[
\det D(q)
=
\det\begin{pmatrix}1&q\\t&1\end{pmatrix}
\cdot
\det\begin{pmatrix}s&q\\1&1\end{pmatrix}
=
(1-qt)(s-q).
\]
For $q\in\Q^\times$, both factors are units in $S=\Q[[s,t]]$: the first
has constant term $1$, and the second has constant term $-q\neq0$.
Hence $\det D(q)$ is a unit, so $D(q)$ is an isomorphism
$U^2\to W^2$. It follows that
$
\ker d(q)|_{N^2}=W^2=\im d(2q)|_{N^2}.
$
Together with Lemma~\ref{lem:B}, this proves that
$$
	\cdots\xrightarrow{d_3}R^2\xrightarrow{d_2}R^2
	\xrightarrow{d_1}R^2\longrightarrow M\longrightarrow0
$$
is a free resolution. Its entries lie in the maximal ideal, so it is
minimal. The decomposition over $S$ gives
$
M\cong \coker(d_1|_{B^2})\oplus\coker(d_1|_{N^2})
\cong S^8\oplus U^2\cong S^{12}
$
as $S$-modules. In particular, $s,t$ is $M$-regular and $M$ is maximal
Cohen--Macaulay.
Applying $\operatorname{Hom}_R(-,R)$ to the resolution  gives
the cochain complex
\(
0\to R^2\xrightarrow{d_1^{\mathsf T}}R^2
\xrightarrow{d_2^{\mathsf T}}R^2
\xrightarrow{d_3^{\mathsf T}}\cdots.
\)
Thus, for $i\geq1$,
$
	\Ext^i_R(M,R)=\ker d_{i+1}^{\mathsf T}/\im d_i^{\mathsf T}.
$
Recall the action of the $x_j$ on $N=U\oplus W$ with $U=Su\oplus Sv$ and
$W=Sw\oplus Sz$. With respect to the ordered bases $(u,v)$ of $U$ and
$(w,z)$ of $W$, each $x_j$ acts as the $2\times2$ matrix
\[
X_1=\begin{pmatrix}1&0\\0&s\end{pmatrix},\quad
X_2=\begin{pmatrix}1&0\\0&1\end{pmatrix},\quad
X_3=\begin{pmatrix}0&1\\1&0\end{pmatrix},\quad
X_4=\begin{pmatrix}0&1\\t&0\end{pmatrix}.
\]

The map $d(q)^{\mathsf T}=\begin{pmatrix}x_1&x_4\\qx_3&x_2\end{pmatrix}$
acts on $R^2$, and on the $N$-summand it sends $U^2\to W^2$ (each $x_j$
maps $U\to W$). In the ordered basis $(u,v,w,z)$ of $U\oplus W$, the
block matrix is
\[
E(q)=\begin{pmatrix}X_1&X_4\\qX_3&X_2\end{pmatrix}
=\begin{pmatrix}
	1&0&0&1\\
	0&s&t&0\\
	0&q&1&0\\
	q&0&0&1
\end{pmatrix}.
\]
This is the transpose of the $2\times2$ matrix
$\begin{pmatrix}x_1&qx_3\\x_4&x_2\end{pmatrix}$ over $R$, restricted to
$N$; it is \emph{not} the ordinary transpose of the $4\times4$ matrix
$D(q)$ over $S$, because the entries $X_j$ are themselves $2\times2$
matrices and transposition does not commute with the block structure.
Apply the simultaneous row and column permutation $\sigma=(1,4,2,3)$,
i.e.\ reorder the index set as $(1',2',3',4')=(1,4,2,3)$. Let $P$ be
the corresponding permutation matrix. The conjugation
$P\,E(q)\,P^{\mathsf T}$ is a similarity transformation, so it preserves
the determinant and the cokernel.
 Reordering the rows of $E(q)$ into
the order (row $1$, row $4$, row $2$, row $3$) gives
\[
P E(q)=\begin{pmatrix}
	1&0&0&1\\
	q&0&0&1\\
	0&s&t&0\\
	0&q&1&0
\end{pmatrix}.
\]

Reordering the columns into the
order (col $1$, col $4$, col $2$, col $3$) gives
\[
P E(q) P^{\mathsf T}
=
\begin{pmatrix}
	1&1&0&0\\
	q&1&0&0\\
	0&0&s&t\\
	0&0&q&1
\end{pmatrix}
=
\begin{pmatrix}1&1\\q&1\end{pmatrix}
\ \oplus\
\begin{pmatrix}s&t\\q&1\end{pmatrix}.
\]

 The
entries $E_{ij}$ with exactly one of $i,j$ in $\{1,4\}$ and the other in
$\{2,3\}$ are
\[
E_{12}=E_{13}=E_{42}=E_{43}=0,
\qquad
E_{21}=E_{31}=E_{24}=E_{34}=0,
\]
as read directly from $E(q)$. Hence the off-diagonal blocks vanish and
the decomposition is exact.
For $q=2^i$ with $i\ge1$, namely, first block is
	$\begin{pmatrix}1&1\\2^i&1\end{pmatrix}$ has determinant
	$1-2^i$, which is a nonzero rational number, hence a unit in
	$S=\Q[[s,t]]$. Therefore this block is invertible and contributes
	nothing to the cokernel.
The second  block is
	$\begin{pmatrix}s&t\\2^i&1\end{pmatrix}$: the entry $1$ in position
	$(2,2)$ is a unit, so invertible row and column operations reduce it to
	diagonal form. Concretely, subtract $t\cdot(\text{row }2)$ from row $1$:
	\(
	\begin{pmatrix}s&t\\2^i&1\end{pmatrix}
	\rightsquigarrow
	\begin{pmatrix}s-2^it&0\\2^i&1\end{pmatrix},
	\)
	then use the unit $1$ to clear the $2^i$ by a column operation:
	\(
	\rightsquigarrow
	\begin{pmatrix}s-2^it&0\\0&1\end{pmatrix}
	=
	\operatorname{diag}(s-2^it,\,1).
	\)
Therefore
$
\coker E(2^i)
\cong
\coker\operatorname{diag}(s-2^it,1)
\cong
S/(s-2^it)\oplus S/(1)
\cong
S/(s-2^it).
$ The determinant is multiplicative across the block decomposition:
\[
\det E(2^i)
=
\det\begin{pmatrix}1&1\\2^i&1\end{pmatrix}
\cdot
\det\begin{pmatrix}s&t\\2^i&1\end{pmatrix}
=
(1-2^i)(s-2^it).
\]
For $i\ge1$, $1-2^i\neq0$, and $s-2^it\neq0$ in the domain
$S=\Q[[s,t]]$. Hence $\det E(2^i)\neq0$. Since $S$ is a domain and
$E(2^i)$ is a square matrix over $S$ with nonzero determinant, the
induced map $S^4\to S^4$, and hence $U^2\to W^2$, is injective.

\begin{claim}	\label{kill}
One has	$d(q)^{\mathsf T}$ kills $W^2$. 
Therefore,	$\ker (d_{i+1}^{\mathsf T}|_{N^2})=W^2$ for $i\geq0$.
\end{claim}

\begin{PROOF}{\ref{kill}}
	Recall $N=U\oplus W$ with $U=Su\oplus Sv$ and $W=Sw\oplus Sz$, and the
	$B$-module structure on $N$ is given by declaring
	\(
	x_j w = x_j z = 0 ,
	\)
	together with $(\ast)$ from Remark \ref{r62}.
	So indeed every $x_j$ maps $U\to W$ and \emph{kills} $W$:
	$x_j\cdot w=x_j\cdot z=0$.
	The map $d(q)=\begin{pmatrix}x_1&qx_3\\x_4&x_2\end{pmatrix}$ acts on
	$R^2=R\oplus R$. Its transpose
	\(
	d(q)^{\mathsf T}=\begin{pmatrix}x_1&x_4\\qx_3&x_2\end{pmatrix}
	\)
	acts on $R^2$ by the same rule but with the matrix transposed. On the
	$N$-summand, using the decomposition $N=U\oplus W$.
The \text{source} $N^2$ splits as
		$(U\oplus W)^2=U^2\oplus W^2\oplus(\text{mixed})$, but since each
		entry $x_j$ kills $W$, any input lying in $W^2$ is sent to $0$.
 The \text{target} $N^2$ likewise splits, and $x_j(U)\subseteq W$,
		so $d(q)^{\mathsf T}$ maps $U^2\to W^2$.
So on $N^2$ the map $d(q)^{\mathsf T}$ is the block matrix
	\(
	\begin{pmatrix}0 & 0\\ E(q) & 0\end{pmatrix}
	\)
	with respect to $N^2=U^2\oplus W^2$: it sends
	$U^2\xrightarrow{E(q)}W^2$ and kills $W^2$.  \end{PROOF}
Because $E(2^i)$ is injective and every $x_j$ kills $W$, the
transposed differential $d_{i+1}^{\mathsf T}$ vanishes on $W^2$ and
$d_i^{\mathsf T}$ has image $E(2^i)U^2$ on the $N$-summand. Hence
\(
\ker (d_{i+1}^{\mathsf T}\big|_{N^2})=W^2\),
\(
\im (d_i^{\mathsf T}\big|_{N^2})=E(2^i)U^2,
\)
so
\[
\Ext^i_R(M,R)\big|_N
=
W^2/E(2^i)U^2
\cong
\coker E(2^i)
\cong
S/(s-2^it).
\]
Since the $B$-summand is exact by Lemma~\ref{lem:B}, this is the entire
Ext module  
\(
\Ext^i_R(M,R)=W^2/E(2^i)U^2\cong S/(s-2^it).
\)
We also check the $R$-action, rather than only the $S$-module structure.
Every class in the last formula has a representative in $W^2\subseteq N^2$.
The ideal $(x_1,x_2,x_3,x_4)B$ kills $W$, and $N\cdot N=0$ in $R$.
Thus $J$ annihilates these Ext modules. This isomorphism
is therefore an $R$-module isomorphism
\(
\Ext^i_R(M,R)\cong R/(J+(s-2^it)R).
\)
The quotient is isomorphic to $\Q[[t]]$, so $\mathfrak p_i$ is prime
and $\Ass_R\Ext^i_R(M,R)=\{\mathfrak p_i\}$.
For $i\ne j$, the image of $s-2^jt$ in $S/(s-2^it)$ is
$(2^i-2^j)t\ne0$, proving that these primes are distinct.
Finally, $\operatorname{Hom}_R(M,R)=\ker (d_1^{\mathsf T})$ is free
of rank $8+4=12$ over $S$, by Lemma~\ref{lem:B} and the injectivity
of $E(2)$. Every associated prime of this nonzero module contracts to
zero in $S$, because it is $S$-torsion-free. As $R/J=S$, the only
such prime is $J$. Hence $\Ass_R\operatorname{Hom}_R(M,R)=\{J\}$.
This completes the proof of Theorem~\ref{thm:main}.\end{proof}

\begin{fact}\label{prop:ordinary-finite}
i)	For every  $\mathfrak a\subseteq R$, every finitely generated
	$R$-module $L$, and every $j\geq0$,  
	$\Ass_R H_{\mathfrak a}^j(L)$ is finite.
	
	ii) $
	\bigcup_{i\geq0}\Ass_R H_J^i(M,R)
	$
	is infinite. 
\end{fact}

\begin{proof}
i) 	If $\mathfrak a\subseteq J$, then $\mathfrak a^3=0$. The functor
	$\Gamma_{\mathfrak a}$ is the identity on all modules, so
	$H_{\mathfrak a}^j(L)=0$ for $j>0$.
Suppose instead that $\mathfrak a\not\subseteq J$. Its image
	$\overline{\mathfrak a}=(\mathfrak a+J)/J$ in
	$R/J=S=\Q[[s,t]]$ is nonzero. Since $S$ is a two-dimensional local
	domain, $\dim S/\overline{\mathfrak a}\leq1$. Every Noetherian local
	ring of dimension at most one has finitely many prime ideals: its
	minimal primes are finite in number, and every other prime is the
	maximal ideal. It follows that
	\(
	V_R(\mathfrak a)\ \cong\
	\Spec\bigl(S/\overline{\mathfrak a}\bigr)
	\)	is finite; this uses that the nilpotent ideal $J$ is contained in
	every prime of $R$.
	Local cohomology is supported in $V_R(\mathfrak a)$, so its associated
	primes form a finite set. Finally,
	$H_{\mathfrak a}^0(L)=\Gamma_{\mathfrak a}(L)$ is a finite submodule of
	$L$.

ii) 
For finite $A$-modules $P,L$, generalized local cohomology is defined by
$
H_{\mathfrak a}^i(P,L)
=\varinjlim_n\Ext_A^i(P/\mathfrak a^nP,L);
$
taking $P=A$ recovers ordinary local cohomology.   For the present ring, $J^3=0$ makes the limit
system constant for $n\geq3$.
Theorem~\ref{thm:main} therefore gives
$ 
H_J^i(M,R)\cong\Ext_R^i(M,R)
\cong R/\mathfrak p_i.
$ 
Consequently,
$
\bigcup_{i\geq0}\Ass_R H_J^i(M,R)
=\{J\}\cup\{\mathfrak p_i:i\geq1\}
$
is infinite.  
\end{proof}

\section{Ext and the number of generators}
Denote the minimal number of generators of  a module by $\mu(-)$. The main result
of this section slightly extends a result of Serre (and others), by using a simple method.

\begin{lemma}\label{beta}
Let $(R,\fm)$ be  local  and $M$ be  such that $p:=\pd(M)<\infty$. Then  $\mu(\Ext^p_R(M,R))=\beta_p(M).$
\end{lemma}

\begin{proof}Recall that $\beta_p(-)$ is the Betti number.
Look at the minimal free resolution of $M$
$$\begin{CD}
0 @>>> R^{\beta_p(M)} @>
X>>R^{\beta_{p-1}(M)}@>>>
\ldots @>>>  R^{\mu(M)} @>>> M@>>> 0,@.
\\
\end{CD}
$$where  $X$ is a  matrix with entries from $\fm$.
Delete $M$ from the right and apply $\Hom_R(-,R)$ we arise to the following complex$$\begin{CD}
0 @>>>R^{\mu(M)} @>
>>  \ldots @ >>> R^{\beta_{p-1}(M)}@>X^t>>  R^{\beta_p(M)}@>>> 0.
\\
\end{CD}
$$ Here, $(-)^{t}$ denotes the transpose. For simplicity, we set $\ell:=\beta_{p-1}$ and $\underline{\textbf{x}}_i:=(x_{i1},\cdots, x_{i\beta_p(M)})$. Thus  $\Ext^p_R(M,R)=\frac{R^{\beta_p(M)}} {(\underline{\textbf{x}}_1,\cdots,\underline{\textbf{x}}_{\ell})}.$
This shows that $\mu(\Ext^p_R(M,R))\leq \mu(R^{\beta_p(M)})$.
One has
$(\underline{\textbf{x}}_1,\cdots,\underline{\textbf{x}}_{\ell})R^{\beta_p(M)}\subset\fm R^{\beta_p(M)}$. Consequently,
$\fm\frac{R^{\beta_p(M)}}{(\underline{\textbf{x}}_1,\cdots,\underline{\textbf{x}}_{\ell})}=\frac{\fm R^{\beta_p(M)}}{(\underline{\textbf{x}}_1,\cdots,\underline{\textbf{x}}_{\ell})}$.
By \cite[Page 35]{mat}, $\mu(E)=\dim_{k}(E/ \fm E)$ for a finitely generated $R$-module $E$. Therefore,  

$$
\begin{array}{ll}
\mu(\Ext^p_R(M,R))=\dim_{k}(\frac{ R^{\beta_p(M)}}{(\underline{\textbf{x}}_1,\cdots,\underline{\textbf{x}}_{\ell})}/ \frac{\fm R^{\beta_p(M)}}{(\underline{\textbf{x}}_1,\cdots,\underline{\textbf{x}}_{\ell})})= \dim_{k}(\frac{R^{\beta_p(M)}}{\fm R^{\beta_p(M)}}) =\mu(R^{\beta_p(M)}) =\beta_p(M).
\end{array}$$
\end{proof}

Let us recalling the following two results:

\begin{theorem}\label{6.1}(Macaulay 1904, Vasconcelos  1967, and Smith 2013). Let $F$ be a field and $I$ be an ideal in $S:=F[X,Y]$ such
	that $S/I$ is a Poincar\'{e} duality algebra. Then $\mu(I)=2$.
\end{theorem}

They proved a little more,  see \cite[Proposition 2.4]{w1}. The  Poincar\'{e} duality algebra is
Gorenstein.

\begin{theorem}\label{6.2}(Serre 1960). Let $S$ be a regular local ring and $I$ be a height two ideal such that
	that $R:=S/I$ is Gorenstein. Then $\mu(I)=2$.
\end{theorem}

The Gorenstein condition implies  that
$\omega_R=\Ext^2_S(R,S)$ is cyclic. In particular,
we slightly extend  theorems \ref{6.2} and \ref{6.1}:

\begin{corollary}(See Peskine-Szpiro \cite[Cor II.5.7]{ps}).
Let $(S,\fm)$ be a Cohen-Macaulay  local  ring and $I$ a  Cohen-Macaulay ideal of height two and of finite projective dimension. Then $\mu(\Ext^2_S(S/I,S))=\mu(I)-1.$
\end{corollary}

\begin{proof}Set $R:=S/I$.
By Auslander-Buchsbaum formula, $$\pd_S(R)=\depth_S( S)-\depth_S( R)=\depth_S( S)-\depth_R (R)=\dim (S)-\dim( S/I)=\Ht(I)=2.$$    By Hilbert-Burch,
there is a  matrix  $X$ with entries from $\fm$ such that the minimal free resolution of $S/I$ is$$\begin{CD}
0 @>>> S^{\mu(I)-1} @>
X>>
 S^{\mu(I)} @>>>  S @>>> S/I@>>> 0.@.
\\
\end{CD}
$$
From this we conclude that $\beta_2(S/I)=\mu(I)-1$.
In view of Lemma  \ref{beta}  $\mu(\Ext^2_S(S/I,S))=\beta_2(S/I)=\mu(I)-1.$  This is what we want to prove.
\end{proof}

\begin{acknowledgement}
I am grateful to Prof.~Wolmer V.~Vasconcelos for his valuable
historical comments.
\end{acknowledgement}



\begin{thebibliography}{99}
 \bibitem{moh}  
M. Asgharzadeh,  \emph{Cohomological splitting, realization, and  finiteness}, arXiv:2101.08187.



 \bibitem{moh2}  
M. Asgharzadeh, A note on cohomological vanishing theorems, arXiv:2401.14133.






\bibitem{comment} M. Auslander, \emph{Comments on the functor Ext}, Topology {\bf{8}} (1969), 151--166.


\bibitem{aus}
M. Auslander,
\emph{Remarks on a theorem of Bourbaki},
Nagoya Math. J. {\bf{27}}, (1966) 361-369.

\bibitem{ab}
M. Auslander, and D. Buchsbaum,  \emph{Invariant factors and two criteria for projectivity of modules}, Trans. AMS. {\bf{104}}, (1962), 516-522.

\bibitem{ar}M.
Auslander  and I. Reiten, \emph{On a generalized version of the Nakayama conjecture},
Proc. AMS  {\bf{52}} (1975),   69-74.


\bibitem{brid}
M. Bridger, \emph{The $R$-modules $\Ext^i_R(M,R)$ and other invariants of $M$}, Thesis (Ph.D.)-Brandeis University, (1967).

\bibitem{BB}
M. Bridger, \emph{
An ideal criterion for torsion freeness}, Proc. AMS {\bf{33}} (1972), 285-291.

\bibitem{BH}
W. Bruns and J. Herzog,  \emph{Cohen-Macaulay rings}, Cambridge University Press, {\bf{39}}, Cambridge, (1998).

\bibitem{Beder}J. Beder,  \emph{The grade conjecture and asymptotic intersection multiplicity}, Proc. AMS. {\bf{142}} (2014), 4065-4077.

\bibitem{e}D. Eisenbud, C. Huneke, and Wolmer Vasconcelos, \emph{Direct
methods for primary decomposition}, Invent. Math.  {\bf{110}}, (1992), 207-235.

\bibitem{eg}
 E.G.  Evans, and P. Griffith, \emph{
Local cohomology modules for normal domains},
J.  LMS  {\bf{19}} (1979),  277-284.


\bibitem{e} E.G. Evans,   P. Griffith,
\emph{The syzygy problem},
Ann. of Math.  {\bf {114}} (1981), 323-333.

\bibitem{ff}
H.B. Foxby, \emph{On the $\mu^i$ in a minimal injective resolution}, Math. Scand. {\bf{29}} (1971), 175-186.

\bibitem{j}J. P.
Jans, \emph{Some generalizations of finite projective dimension},
Illinois J. Math.  {\bf{5}} (1961), 334-344.



\bibitem{GOR}
S. Greco, M.G. Marinari, \emph{Nagata's criterion and openness of loci for Gorenstein and
complete intersection}, Math. Z. {\bf160} (1978),  207-216.

\bibitem{41}
A. Grothendieck (notes by R. Hartshorne), \emph{Local cohomology} (LC), Springer LNM. {\bf {41}}, Springer-Verlag,  (1966).

\bibitem{sga2}
A. Grothendieck,
\emph{Cohomologie locale des faisceaux coh\'{e}rents et th\'{e}oremes de Lefschetz locaux et globaux} (SGA 2).
S\'{e}minaire de G\'{e}om\'{e}trie Alg\'{e}brique du Bois Marie  1962. Amsterdam: North Holland Pub. Co. (1968).



\bibitem{hh}D.
Hanes and C. Huneke, \emph{
	Some criteria for the Gorenstein property}, 
J. Pure Appl. Algebra {\bf{201}} (2005), no. 1-3, 4--16.

\bibitem{seg}D.A Jorgensen, L.M. Şega, \emph{
Nonvanishing cohomology and classes of Gorenstein rings},
Advances in Mathematics {\bf {188}} (2), 470--490.

\bibitem{mac}F. S. Macaulay,
\emph{On a method of dealing with the intersections of plane curves},
Trans. AMS. {\bf {5}} (1904), 385-410.


\bibitem{mat}
H. Matsumura, \emph{Commutative ring theory}, Cambridge Studies in Advanced Math, \textbf{8}, (1986).

\bibitem{n} R. J. Nunke, \emph{Modules of extensions over Dedekind rings}, Illinois J. Math., {\bf {3}}
(1959),   222-241.

\bibitem{ps}
C. Peskine and L. Szpiro, \emph{Dimension projective finie et cohomologie locale}, IHES. Publ. Math. {\bf
42}, (1973),  49-119.

\bibitem{ss}P. Schenzel, \emph{Cohomological annihilators}, Math. Proc. Cambridge Philos. Soc. {\bf91} (1982),  345-350.

\bibitem{ser}
Jean-Pierre Serre,
\emph{Sur les modules projectifs},
S\'{e}minaire Dubreil. Alg\`{e}bre et theorie des nombres (1960-1961)
{\bf {14}}, Issue: 1,  1--16.




\bibitem{smith}
L.  Smith,  \emph{Hilbert-Kunz invariants and Euler characteristic polynomials}, Pacific J. Math.  {\bf {262}} (2013), 227-255.

\bibitem{vas}
Wolmer V. Vasconcelos,
\emph{Computational methods in commutative algebra and algebraic geometry}, Algorithms and Computation in Mathematics
{\bf{2}} Springer-Verlag, Berlin, (1998).

\bibitem{vas1}
Wolmer V. Vasconcelos,
\emph{The homological degree of a module}, Trans. AMS. {\bf350} (1998),  1167-1179.


\bibitem{w1}
Wolmer V. Vasconcelos, \emph{Ideals generated by $R$-sequences}, J. Algebra {\bf6} (1967),
309-316.


\end{thebibliography}
\end{document}